\documentclass[10pt]{amsart}
\usepackage[T1]{fontenc}
\usepackage{amssymb}
\usepackage{amsmath}
\usepackage[all]{xy}
\usepackage{graphicx}
\usepackage{multicol}
\usepackage{xcolor}
\usepackage{hyperref}

\theoremstyle{plain}
\newtheorem{thm}{Theorem}
\newtheorem{coro}[thm]{Corollary}
\newtheorem*{thmA}{Theorem A}
\newtheorem*{thmB}{Theorem B}
\newtheorem*{thmC}{Theorem C}
\newtheorem*{thmD}{Theorem D}
\newtheorem*{thmE}{Theorem E}
\newtheorem*{thmF}{Theorem F}
\newtheorem*{conv}{Convention}
\newtheorem{lem}[thm]{Lemma}

\newtheorem{prop}[thm]{Proposition}
\newtheorem{conj}[thm]{Conjecture}
\newtheorem{ass}[thm]{Standing Hypothesis}

\theoremstyle{definition}
\newtheorem{ques}[thm]{Question}
\newtheorem{ex}[thm]{Example}
\newtheorem{defn}[thm]{Definition}

\theoremstyle{remark}

\newtheorem{rem}[thm]{Remark}

\numberwithin{thm}{section}
\numberwithin{equation}{section}

\newcommand{\F}{\mathbb{F}}
\newcommand{\Z}{\mathbb{Z}}
\newcommand{\N}{\mathbb{N}}
\newcommand{\Q}{\mathbb{Q}}
\newcommand{\R}{\mathbb{R}}
\newcommand{\C}{\mathbb{C}}
\newcommand{\K}{F}

\newcommand{\ext}{\mathrm{Ext}}
\newcommand{\Span}{\mathrm{span}}
\newcommand{\trg}{\mathrm{trg}}
\newcommand{\ind}{\mathrm{ind}}

\newcommand{\dbl}{[\![}
\newcommand{\dbr}{]\!]}

\newcommand{\dbml}{\langle\!\langle}
\newcommand{\dbmr}{\rangle\!\rangle}
\newcommand{\gr}{{\rm gr}}

\newcommand{\sU}{{\mathcal U}}
\newcommand{\FpG}{\F_p[G]}
\newcommand{\FppG}{\F_p\dbl G\dbr}
\newcommand{\FpX}{\F_p\langle X\rangle}
\newcommand{\FppX}{\F_p\dbml X \dbmr}

\renewcommand{\ker}{\mathrm{Ker}}

\newcommand{\Hom}{\mathrm{Hom}}
\newcommand{\Ext}{\mathrm{Ext}}

\newcommand{\dd}{\mathrm{d}}

\newcommand{\slot}{\,\underbar{\;\;}\,}
\newcommand{\gen}[1]{\langle #1\rangle}

\DeclareMathOperator{\Gal}{Gal}

\begin{document}

\title[Koszul algebras and quadratic duals in Galois cohomology]
 {Koszul algebras and quadratic duals  \\ in Galois cohomology}

\author{Jan~Min\'a\v{c}}
\author{Federico~Pasini}
\author{Claudio~Quadrelli}
\author{Nguy\~{\^e}n~Duy~T\^{a}n}

\dedicatory{To the memory of Jean-Louis Koszul.}

\date{\today}

\address{Department of Mathematics \\ University of Western Ontario \\ Middlesex College, London ON, Canada}
\email{minac@uwo.ca}
\address{Department of Applied Mathematics \\ University of Western Ontario \\ Middlesex College, London ON, Canada}
\email{fpasini@uwo.ca}
\address{Department of Mathematics and Applications\\ University of Milano-Bicocca\\  Via R.Cozzi 55 -- Ed.~U5, Milan, Italy-EU}
\email{claudio.quadrelli@unimib.it}
\address{School of Applied Mathematics and Informatics \\ Hanoi University of Science and Technology \\ 1 Dai Co Viet Road, Hanoi, Vietnam}
\email{tan.nguyenduy@hust.edu.vn}

\begin{abstract}
We investigate the Galois cohomology of finitely generated maximal pro-$p$ quotients of absolute Galois groups.
Assuming the well-known conjectural description
of these groups, we show that Galois cohomology has the PBW property. Hence in particular it is a Koszul algebra.
This answers positively a conjecture by Positselski in this case. We also provide an analogous unconditional result
about Pythagorean fields. Moreover, we establish some results that relate the quadratic dual of Galois cohomology
with the $p$-Zassenhaus filtration on the group. This paper also contains a survey of Koszul property in Galois cohomology
and its relation with absolute Galois groups.
\end{abstract}

\subjclass[2010]{Primary 12G05; Secondary 17A45, 20F40, 20F14}


\thanks{The first author is partially supported by the Natural Sciences and Engineering Research Council of Canada grant R0370A01.
The third author was partially supported by the Israel Science Fundation grant no.~152/13 and by a fellowship ``Giovani Ricercatori Protagonisti'' funded by Fondazione~CR Firenze.
The fourth author is partially supported by the Vietnam National Foundation for Science and Technology Development (NAFOSTED) under grant number { 101.04-2019.314.}}

\maketitle

\section{Introduction}\label{s:intro}
Let $p$ be a prime number and $\K$ be a field. The \emph{absolute Galois group} $G_{\K}=\Gal(\K_{\mathrm{sep}}/\K)$ is the Galois group of the maximal separable extension $\K_{\mathrm{sep}}$ of $\K$. This group contains all the Galois-theoretic information about $\K$, as any finite Galois group of an extension of $\K$ is a quotient of $G_\K$.

Despite such a central role in Galois theory, absolute Galois groups of fields remain rather mysterious.
They are profinite groups, and one of the main problems in current research in Galois theory is
to determine which profinite groups are absolute Galois groups of fields.
One way to investigate the structure of profinite groups is Galois cohomology.
This is a cohomology theory based on continuous cochains and coboundaries (cf.~\cite[Ch.~2]{serre}).
The Galois cohomology of $\K$ with coefficients in the finite field $\F_p=\Z/p\Z$ is a graded algebra
\[H^\bullet(G_{\K},\F_p)=\bigoplus_{n\geq0}H^n(G_{\K},\F_p),\]
with respect to the homological degree and the graded-commutative cup product (cf.~\cite[Ch.~1, \S~4]{nsw:cohm})
\[\cup \colon H^r(G_{\K},\F_p)\times H^s(G_{\K},\F_p)\longrightarrow H^{r+s}(G_{\K},\F_p),\qquad r,s\geq0.\]

Throughout the paper, we adopt the following
\begin{ass}\label{ass}
For $p$ odd, $\K$ contains a primitive $p$\textsuperscript{th} root of unity $\zeta_p$.
\end{ass}
Some results require a stronger assumption in the case $p=2$:
\begin{ass}\label{ass2}
For $p=2$, $\K$ contains $\sqrt{-1}$.
Equivalently, in purely group-theoretic terms, the image of the \emph{orientation} associated to a pro-$2$ group $G$ is contained in $\{1+4\lambda\mid \lambda\in\Z_2\}$ (see~Definition~\ref{defn:cyclochar}).
\end{ass}

Recently, M.~Rost and V.~Voevodsky completed the proof of the long-standing Bloch-Kato conjecture
(cf.~\cite{suslin:norm, haeseweibel, voevodsky:motivic, voevodsky:BK}, see also \cite[\S~VI.4]{nsw:cohm}).
Under our Hypothesis \ref{ass} this theorem implies that $H^\bullet(G_{\K},\F_p)$
is generated by elements of degree $1$ and its relations are generated by homogeneous relations of degree $2$.
Such algebras are called \emph{quadratic algebras} (cf.~\cite[\S~1.2]{PoliPosi}).

Locally finite-dimensional (cf.~\cite[Page 1]{PoliPosi} and Definition~\ref{defi:algebras})
quadratic algebras come equipped with a duality.
The \emph{quadratic dual} of such an algebra $A_\bullet$ is a quadratic algebra generated over the same field
by the dual space of the space of generators of $A_\bullet$.
Its relators
form the orthogonal complement of the space of the relators of $A_\bullet$ (see Definition~\ref{defi:quadratic algebras}).
The double quadratic dual $(A_\bullet^!)^!$ is isomorphic to $A_\bullet$ as a graded algebra.

Among quadratic algebras, the significant class of Koszul algebras has been singled out in \cite{priddy}. 
These are the graded algebras for which the ground field has a linear free resolution as trivial module,
that is a resolution of free graded modules $(P_i,d_i)$ with each $P_i$ generated in degree $i$.
They are characterised by an uncommonly nice cohomological behaviour (cf.~\S~\ref{sec:quadratic}).

For a general graded algebra $A_\bullet=\bigoplus_{n\geq0}A_n$ over a field $\Bbbk$, the $\Bbbk$-cohomology
is defined as the direct sum of the derived functors of the functor $\Hom_{A_\bullet}(\slot,\Bbbk)$ evaluated on $\Bbbk$, that is $\ext_{A_\bullet}(\Bbbk,\Bbbk)=\bigoplus_{i\in\N}\ext_{A_\bullet}^i(\Bbbk,\Bbbk)$.
Finding an explicit description of the cohomology of an arbitrary graded algebra is generally a hopeless battle. But the cohomology of a Koszul $\Bbbk$-algebra $A_\bullet$ is just its \emph{quadratic dual} $A_\bullet^!$. This also means that Koszul algebras can be reconstructed from their cohomology. There is not the typical loss of information in passing from an algebraic or topological object to its cohomology.

Koszul algebras arise in various fields of mathematics, such as representation theory, noncommutative algebra,
quantum groups, noncommutative geometry, algebraic geometry and topology
(cf., e.g., \cite{begiso,froberg75,kohno85,kohno83,lodval}). 
Moreover, Koszul algebras have been studied in the context of Galois cohomology by L.~Positselski and A.~Vishik
(cf.~\cite{posivis:koszul}).
This was considerably extended in further remarkable papers by Positselski
(cf.~\cite{positselsky:koszul,positselsky:koszul2}).
In particular, in \cite[\S~0.1]{positselsky:koszul2} he explicitly formulated a general conjecture, of which the following is a special case.

\begin{conj}[Positselski]
 Let $\K$ be a field containing a primitive $p$\textsuperscript{th} root of unity.
Then the algebra $H^\bullet(G_{\K},\F_p)$ is Koszul.
\end{conj}

In \cite[Theorem~1.3]{positselsky:koszul} it is shown that the \emph{Koszulity} --- i.e., the property of being Koszul --- of the reduction modulo $p$
of the Milnor K-theory, along with the bijectivity of the corresponding norm residue map in degree $2$
and the injectivity of the one in degree $3$ would give an easier and natural proof of Bloch-Kato conjecture for the prime $p$. Positselski's Conjecture can be seen as a strenghtening of Bloch-Kato Conjecture. %

In this paper we prove that in several cases the Galois cohomology and its quadratic dual have the stronger PBW property (cf.~\cite[Chapter 4]{lodval}).

There is a conjectural description of maximal pro-$p$ quotients of absolute Galois groups that are finitely generated and satisfy our Hypothesis \ref{ass}. This description roughly says that all such groups can be constructed inductively from finitely generated free pro-$p$ groups and Galois groups of maximal $p$-extensions of local fields using free products and certain semidirect products.
Slightly more general groups constructed this way are called \emph{elementary type} pro-$p$ groups (see \S~\ref{ssec:ET groups} for details) and the former conjecture is known as Elementary Type Conjecture,
in short ETC, formulated by I.~Efrat (cf.~\cite{ido:ETC}).

Our main results are the following.

\begin{thmA}
If $G$ is an elementary type pro-$p$ group, then $H^\bullet(G,\F_p)$ is PBW, and hence Koszul.
\end{thmA}

\begin{rem}
In Theorem A we can replace $G$ by any absolute Galois group $G_F$ such that its maximal pro-$p$ quotient $G_F(p)$ is of elementary type. This follows from the Hochschild-Serre spectral sequence associated to the exact sequence
\[
\xymatrix{
1\ar[r] & \ker(\pi)\ar[r] & G_F\ar^-\pi[r] & G_F(p)\ar[r] & 1,
}
\]
where $\pi$ is the canonical projection.
Indeed, $\ker(\pi)\cong\Gal(F_{\mathrm{sep}}/F(p))$, where $F(p)$ is the compositum
of all Galois extensions of $F$ of degree a power of $p$.
By Bloch-Kato Conjecture, $H^n(\Gal(F_{\mathrm{sep}}/F(p)),\F_p)=0$ for all $n>0$,
so the Hochschild-Serre spectral sequence collapses at the second page.
As a consequence, the inflation map $\mathrm{inf}\colon H^\bullet(G_F(p),\F_p)\to H^\bullet(G_F,\F_p)$ is an isomorphism.
(This argument is well-known to the experts, cf., e.g., \cite[Proposition~5.9, p.~45]{AEJ} and \cite[\S~24.3]{ido:miln}.)
\end{rem}

The quadratic dual of cohomology is related to the successive factors of a descending series on $G$,
the $p$-Zassenhaus filtration (see Definition~\ref{defi:zass}).
Let $\sU(L(G))$ be the restricted universal enveloping algebra of the restricted Lie algebra $L(G)$
given by the $p$-Zassenhaus filtration on $G$ (see \S~\ref{ssec:restricted}). 

\begin{thmB}
If $G$ is an elementary type pro-$p$ group, satisfying Hypothesis \ref{ass2} if $p=2$, then the quadratic dual of $H^\bullet(G,\F_p)$ is
$\sU(L(G))$ and it is PBW, and hence Koszul.
\end{thmB}

By a result of Jennings (Theorem \ref{thm:LG grFpG}), for a finitely generated pro-$p$ group $G$,
$\sU(L(G))$ is isomorphic to the graded object $\gr\FppG$ associated to the filtration
on the complete group algebra of $G$ by the successive powers of the complete augmentation ideal
(see Definition~\ref{defi:FpG}).
Thus, Theorem~B gives a partial positive answer to a question of T.~Weigel (cf.~\cite{thomas:proc}):

\begin{ques}[Weigel]
Let $\K$ be a field containing a primitive $p$\textsuperscript{th} root of unity.
Is the algebra $\sU(L(G_\K(p)))$ Koszul?
\end{ques}


Theorem B can be proved without Hypothesis \ref{ass2} for the class of Demushkin groups (see \S~\ref{sec:ET}). In particular, this applies to maximal pro-$p$ quotients of absolute Galois groups of local fields that contain a primitive $p$\textsuperscript{th} root of unity. In combination with Theorem A, this gives:
\begin{thmC}
If $G$ is a Demushkin pro-$p$ group, then the algebras $H^\bullet(G,\F_p)$ and $\sU(L(G))$ are quadratic dual to each other and PBW, hence Koszul.
\end{thmC}


There are numerous families of fields for which a suitable version of the Elementary Type Conjecture is known to hold. For each of these families, the results of this paper guarantee the Koszulity of Galois cohomology unconditionally.
In many cases, also a description of $\sU(L(G_F))$ is available and the duality of the two algebras can be proved.
The notion of Pythagorean field will be recalled in \S~\ref{ssec:pita}, the notion of $p$-rigid field in \S~\ref{ssec:prigid}. Pseudo algebraically closed fields (PAC fields) are studied in \cite[Chapter 11]{friedjarden}.

\begin{thmD}
Let $F$ be a field such that the quotient $F^\times/(F^\times)^p$ is finite.
The algebra $H^\bullet(G_F,\F_p)$ is PBW, and hence Koszul,
in the following cases:
\begin{itemize}
 \item[(a)] $F$ is finite;
 \item[(b)] $F$ is a pseudo algebraically closed (PAC) field, or an extension of relative trascendence degree 1 of a PAC field;
 \item[(c)] $F$ is an algebraic extension of $\Q$;
 \item[(d)] $F$ is an extension of trascendence degree 1 of a local field;
 \item[(e)] $F$ is $p$-rigid.
\end{itemize}
Moreover, if $p=2$ the above also holds in the cases:
\begin{itemize}
\item[(f)] $F$ is algebraic extension of a global field of characteristic not 2;
\item[(g)] $F$ is a Pythagorean field.
\end{itemize}
In cases (a) to (f), $\sU(L(G_F))$ is quadratic dual to $H^\bullet(G_F,\F_p)$ and hence it is Koszul as well.
\end{thmD}

As Positselski proved in \cite{positselsky:koszul2}, the cohomology of local and global fields also has PBW property. These facts corroborate a strenghtening of Positselski's Conjecture.
\begin{conj}
 Let $\K$ be a field containing a primitive $p$\textsuperscript{th} root of unity.
Then the algebra $H^\bullet(G_{\K},\F_p)$ is PBW.
\end{conj}

Koszulity of an algebra can be paraphrased as the fact that there are not surprising ways in which an element
of the algebra can turn to be $0$, or in other words ``only what is expected to be $0$ is actually $0$''.
Our next results go in the same philosophical direction.
When a minimal presentation $G=\gen{x_1,\dots,x_d\mid r_1,\dots,r_m}$ of a finitely generated pro-$p$ group is given
(see \S~\ref{sec:pairings and duals}), commutators and powers applied to a given relator $r_i$ provide automatically
some higher degree relators of $\gr\FppG$. Let us say that these are the ``predictable'' higher degree relations. 
But a priori, $\gr\FppG$ might be subject to other, ``unpredictable'' relations. Under the mild assumption
that Galois cohomology is quadratic, the predictable relations are orthogonal to the relations in Galois cohomology:

\begin{thmE}
Let $G=\gen{x_1,\dots,x_d\mid r_1,\dots,r_m}$ be a minimal presentation of a finitely generated pro-$p$-group $G$,
and assume that $H^\bullet(G,\F_p)$ is quadratic.
Then there is an isomorphism of quadratic $\F_p$-algebras
\[
 H^\bullet(G,\F_p)^!\cong \frac{\F_p\langle X\rangle}{\mathcal{R}},
\]
where $\mathcal{R}$ is the two-sided ideal generated by the \emph{initial forms}
(see the end of \S~\ref{ssec:restricted}) of the defining relations $r_1,\dots,r_m$.
\end{thmE}

Under the same hypothesis, unpredictable relations cannot exist in lower degrees and do not exist at all
for the class of quadratically defined pro-$p$ groups introduced in \cite{KLM}
(see Definition~\ref{def:quadratic presentation}).

\begin{thmF}
Let $G$ be a finitely generated pro-$p$ group with $H^\bullet(G,\F_p)$ quadratic. 
Then there is an epimorphism of graded $\F_p$-algebras
\[
\xymatrix{H^\bullet(G,\F_p)^!\ar@{->>}[r] & \gr\FppG}
\]
which is an isomorphism in degrees $0$, $1$, $2$.
Further, if $G$ has a quadratically defined presentation, then the previous map is an isomorphism in all degrees.
\end{thmF}
Thus, the following question arises naturally.
\begin{ques}
Does every maximal pro-$p$ quotient of an absolute Galois group admit a quadratically defined presentation?
\end{ques}

The structure of the paper is as follows.
In \S~\ref{sec:quadratic} we provide a survey of quadratic algebras, focusing on Koszul and PBW properties.
\S~\ref{sec:zass} is devoted to some preliminaries on filtrations and algebras associated to groups in general
and pro-$p$ groups in particular.
In \S~\ref{sec:ET} we recall the definition of the class of elementary type pro-$p$ groups.
Theorems A, B and C are proved in \S~\ref{sec:A,B,C}, Theorem D in \S~\ref{sec:D}.
\S~\ref{sec:pairings and duals} contains the relevant facts about pairings, quadratic duals and
filtrations in Galois cohomology that lead to the proof of Theorems E and F. 

\subsection*{Acknowledgment}
{\small Our thinking about Koszul property of Galois cohomology was deeply influenced by J.P.~Labute's work on Demushkin groups, mild groups and filtrations in group theory. It has been a great privilege of the first author to be able to personally collaborate with him and to discuss a number of related topics. We thank Labute also for his encouragement and enthusiasm related to our work. We learned and are still learning about quadratic algebras from inspiring papers by A.~Polishchuk and L.~Positselski, as well as their extremely useful and inspiring book \cite{PoliPosi}. We are grateful to Th.S.~Weigel for his cheerful friendship and his guidance throughout the years, especially while acting as the PhD supervisor of the second author and as the co-supervisor (along with the first author) of the third author. We thank D.M.~Riley for many exciting discussions and valuable suggestions. It is also a pleasure to thank our friends, including S.K.~Chebolu, M.~\v{C}izek, I.~Efrat, S.~Gille, C.~Hall, R.~Libman, E.~Matzri, D.~Neftin, M.~Palaisti, O.~Puglisi, M.~Rogelstad, A.~Schultz, and I.~Snopce, who shared the time and the excitement of this research with the authors, offering insightful advice and words of encouragement.

Last but not least we thank warmly the referee for reading our paper carefully and for providing a number of helpful suggestions which improved and clarified our exposition.}

\section{Quadratic algebras and the Koszul property}\label{sec:quadratic}

This section provides the basic definitions, notable facts and cohomological background for the classes of Koszul algebras and PBW algebras.
In this section $\Bbbk$ denotes an arbitrary field. 

\begin{defn}\label{defi:algebras}
An associative, unital $\Bbbk$-algebra $A$ is \emph{graded} if it decomposes as the direct sum of $\Bbbk$-vector
spaces $A_\bullet=\bigoplus_{i\in\Z}A_i$ such that, for all $i,j\in\Z$, $A_iA_j\subseteq A_{i+j}$.
The summand $A_i$ is the \emph{homogeneous component} of degree $i$.
A graded $\Bbbk$-algebra $A_\bullet$ is \emph{connected} if $A_i=0$ for all $i<0$ and $A_0=\Bbbk\cdot 1_A\cong \Bbbk$;
\emph{generated in degree $1$} if all its (algebra) generators are in $A_1$;
\emph{locally finite-dimensional} if $\dim_{\Bbbk} A_i<\infty$ for all $i\in\Z$ (another common terminology for this is
\emph{finite-type}).
A module $M$ over a graded algebra $A_\bullet$ is \emph{graded} if it decomposes as the direct sum
of $\Bbbk$-vector spaces $M_\bullet=\bigoplus_{i\in\Z}M_i$ s.t., for all $i,j\in\Z$, $A_iM_j\subseteq M_{i+j}$.
A graded $A_\bullet$-module $M_\bullet$ is \emph{finite-type} if $\dim_{\Bbbk} M_i<\infty$ for all $i\in\Z$.
\end{defn}

For a vector space $V$ over $\Bbbk$, let $T_\bullet(V)$ denote the \emph{tensor} $\Bbbk$-algebra generated by $V$, i.e.,
\[T_\bullet(V)=\bigoplus_{n\geq0}V^{\otimes n},\qquad \text{with }V^{\otimes0}=\Bbbk.\]
It is a graded algebra with the tensor product as algebra multiplication. We will systematically omit tensor signs, and simply use juxtaposition, in writing the elements of a tensor algebra. In particular, if $V$ is finite-dimensional, with a basis $\{x_1,\dots,x_n\}$, then $T_\bullet(V)$ can (and will) be identified with the algebra $\Bbbk\gen{x_1,\dots,x_n}$ of noncommutative polynomials in the variables $x_1,\dots,x_n$, graded by polynomial degree.

Ignoring the grading, tensor algebras are the free objects in the category of (associative, unital) algebras (cf.~\cite[\S 1.1.3]{lodval}), so any algebra is a quotient of a tensor algebra. When an algebra is presented as a quotient of a tensor algebra, we will simply identify an element of the former with each of its representatives in the latter.

\begin{ass}\label{ass:algebra hypotheses}
Henceforth we assume graded algebras to be locally finite-dimensional, connected and generated in degree $1$,
unless otherwise specified. In particular, every such algebra $A_\bullet$ is equipped with
the \emph{augmentation map} $\varepsilon\colon A_\bullet\to\Bbbk$ that is the projection onto $A_0$.
Its kernel is the \emph{augmentation ideal} $A_+=\bigoplus_{n\geq1}A_n$.
The augmentation map gives $\Bbbk$ a canonical structure of graded $A_\bullet$-module concentrated in degree $0$,
via the action $a\cdot k=\varepsilon(a)k$.

In the following, ground fields of algebras will be always understood to be equipped with this module structure.
\end{ass}

\subsection{Quadratic algebras}\label{ssec:quadratic alg}
\label{ssec:quadratic algebras}
For a finite-dimensional $\Bbbk$-vector space $V$, let $V^*=\Hom_{\Bbbk}(V,\Bbbk)$ be the $\Bbbk$-dual space of $V$.
With a little abuse, we will then systematically identify $(V\otimes V)^*=V^*\otimes V^*$.

\begin{defn}\label{defi:quadratic algebras}
A \emph{quadratic algebra} is a quotient
\[
 A_\bullet=\frac{T_\bullet(V)}{(\Omega)},
\]
of the tensor algebra over some $\Bbbk$-vector space $V$, with $(\Omega)$ the two-sided ideal generated by a vector subspace $\Omega\subseteq V\otimes V$. 
Throughout this paper, we further require $V$ to be finite-dimensional, in accordance with the Standing Hypothesis \ref{ass:algebra hypotheses}. In other words, a quadratic algebra is a connected graded algebra that admits a presentation with finitely many degree-$1$ generators
and homogeneous degree-$2$ relations.
The component $V=A_1$ is the \emph{space of generators} of $A_\bullet$, $\Omega$ is the \emph{space of relators}, and we will use the
notation $A=Q(V,\Omega)$.

The \emph{quadratic dual} $A_\bullet^!$ of a quadratic $\Bbbk$-algebra $A_\bullet$
is the quadratic $\Bbbk$-algebra $A_\bullet^!=T_\bullet(V^*)/(\Omega^\perp)$, with 
\[\Omega^\perp=\{f\in (V\otimes V)^*\:|\:f(\omega)=0\text{ for all }\omega\in \Omega\}\subseteq V^*\otimes V^*.\]
\end{defn}
\begin{rem}
For quadratic algebras, the natural grading on $T_\bullet(V)$ passes to the quotient thanks to the homogeneity
of the relations. This is the way quadratic algebras are understood to be graded.
\end{rem}
Thanks to finite-dimensionality assumptions, one has $(A_\bullet^!)^!=A_\bullet$.

\begin{ex}\label{ex:quadratic}
Let $V$ be a finite-dimensional $\Bbbk$-vector space, with a basis $B=\{x_1,\dots,x_d\}$. Let $B^*=\{x_1^*,\ldots,x^*_d\}$ be the dual basis to $B$.
\begin{itemize}
\item[(a)] The tensor algebra $T_\bullet (V)$ is quadratic, with trivial space of relators.
Its quadratic dual is the \emph{trivial} algebra on $V^*$, that is $\Bbbk\oplus V^*$ with trivial multiplication,
whose space of relators is $V^*\otimes V^*$.
\item[(b)] The symmetric algebra 
\[ \mathcal{S}_\bullet(V)=\frac{T_\bullet (V)}{( x_i x_j-x_j x_i\mid x_i,x_j\in B)}\]
and the exterior algebra
\[ {\Lambda}_{\bullet}(V^*)=\frac{T_\bullet(V^*)}{( x^*_i x^*_i, x^*_i x^*_j+x^*_j x^*_i\mid x^*_i,x^*_j\in B^*)}\]
are quadratic and dual to each other (cf.~\cite[Examples~3.2.4]{lodval}).
We shall identify the symmetric algebra $\mathcal{S}_\bullet(V)$ with the commutative polynomial algebra
$\Bbbk[x_1,\dots,x_d]$; the grading induced by $T_\bullet(V)$ coincides with polynomial degree.
\end{itemize}
\end{ex}

\begin{ex}[Constructions of quadratic algebras]\label{ex:quadratic2}
Let $A_\bullet$ and $B_\bullet$ be two quadratic algebras, with spaces of relators $\Omega_A$ and $\Omega_B$, respectively.
From $A_\bullet$ and $B_\bullet$ one may construct new quadratic algebras as follows (cf.~\cite[\S~3.1]{PoliPosi}).
\begin{itemize}
\item[(a)] The \emph{direct sum} $C_\bullet=A_\bullet\sqcap B_\bullet$ of $A_\bullet$ and $B_\bullet$
 is the connected quadratic algebra $C_\bullet$ such that $C_n=A_n\oplus B_n$ for every $n\geq1$ 
and such that $ab=ba=0$ for every $a\in A_{+},b\in B_{+}$.
In other words,
\[ C_\bullet=\frac{T_\bullet(A_1\oplus B_1)}{(\Omega)},\qquad 
\text{with }\Omega=\Omega_A\oplus \Omega_B\oplus(A_1\otimes B_1)\oplus(B_1\otimes A_1). \]
 \item[(b)] The \emph{free product} $C_\bullet=A_\bullet\sqcup B_\bullet$ of $A_\bullet$ and $B_\bullet$
is the connected quadratic algebra $C_\bullet$ with $C_1=A_1\oplus B_1$ and
\[ C_\bullet=\frac{T_\bullet(A_1\oplus B_1)}{(\Omega)},\qquad 
\text{with }\Omega=\Omega_A\oplus \Omega_B.\]
 \item[(c)] The \emph{symmetric tensor product} $C_\bullet=A_\bullet\otimes^1 B_\bullet$ of $A_\bullet$
 and $B_\bullet$
is the quadratic algebra $C_\bullet$ with $C_1=A_1\oplus B_1$ and
\[ C_\bullet=\frac{T_\bullet(A_1\oplus B_1)}{( \Omega)},\qquad \text{with }\Omega=\Omega_A\oplus \Omega_B\oplus I, \]
where $I$ is the subspace generated by the elements $ab-ba$ with $a\in A_1$ and $b\in B_1$.
 \item[(d)] The \emph{skew-symmetric tensor product} $C_\bullet=A_\bullet\otimes^{-1} B_\bullet$  of $A_\bullet$
 and $B_\bullet$
is the quadratic algebra $C_\bullet$ with $C_1=A_1\oplus B_1$ and 
\[C_\bullet=\frac{T_\bullet(A_1\oplus B_1)}{( \Omega)},\qquad \text{with }\Omega=\Omega_A\oplus \Omega_B\oplus I,\]
where $I$ is the subspace generated by the elements $ab+ba$ with $a\in A_1$ and $b\in B_1$.
\end{itemize}
\end{ex}

\begin{rem}\label{rem:de Morgan}
The following De Morgan dualities hold (cf.~\cite[Chapter 3, Corollary 1.2]{PoliPosi}):
\[
\begin{array}{rcl}
(A_\bullet\sqcap B_\bullet)^!&=&A^!_\bullet\sqcup B^!_\bullet;\\
(A_\bullet\sqcup B_\bullet)^!&=&A^!_\bullet\sqcap B^!_\bullet;\\
(A_\bullet\otimes^1 B_\bullet)^!&=&A^!_\bullet\otimes^{-1} B^!_\bullet;\\
(A_\bullet\otimes^{-1} B_\bullet)^!&=&A^!_\bullet\otimes^1 B^!_\bullet.
\end{array}
\]
\end{rem}



\subsection{Cohomology of algebras}\label{ssec:cohomology of alg}
\label{ssec:graded}
This subsection, the next and part of \S \ref{ssec:PBW prop} are intended as a reference for the interested reader about the cohomological background and definition of
Koszul algebras.
The reader who wishes to proceed directly to the main results can safely skip to Remark \ref{rem:PBW} in \S \ref{ssec:PBW prop}, bearing in mind that the PBW property, which is what we actually prove for Galois cohomology of elementary type pro-$p$ groups, implies Koszulity.

Let $L$ and $M$ be modules over the graded $\Bbbk$-algebra $A_\bullet$. Let $\Hom_{A_\bullet}(L,M)$ denote the space of $A_\bullet$-module homomorphisms from $L$ to $M$.
For every left $A_\bullet$-module $M$, the contravariant functor $\Hom_{A_\bullet}(\slot,M)$ is left-exact
but in general not exact, hence it has right derived functors $\Ext_{A_\bullet}^i(\slot,M)$. 
These are defined as follows: for another $A_\bullet$-module $L$, choose a projective resolution $(P_i, d_i)$ of $L$;
then $\ext_{A_\bullet}^i(L,M)$ is the $i$th cohomology group of the complex $(\Hom_{A_\bullet}(P_i,M), d_{i+1}^*)$.

The cohomology $H^\bullet(A_\bullet,\Bbbk)$ of $A_\bullet$ is defined to be the graded $\Bbbk$-module
\[\ext_{A_\bullet}^\bullet(\Bbbk,\Bbbk)=\bigoplus_{i\in\N}\Ext_{A_\bullet}^i(\Bbbk,\Bbbk).\]
The terms $\Ext_{A_\bullet}^i(\Bbbk,\Bbbk)$ may be computed as the cohomology of the \emph{normalized cobar complex}
\[
Cob^i(A_\bullet)=\Hom_{A_\bullet}(A_\bullet\otimes A_+^{\otimes i}\otimes \Bbbk,\Bbbk)\cong(A_+^{\otimes i})^*,
\]
 with differential $\delta^i\colon (A_+^{\otimes i})^*\to (A_+^{\otimes i+1})^*$ 
defined by \[\delta^{i}(\varphi)(a_1\otimes\dots\otimes a_{i+1})=\sum_{k=1}^{i}(-1)^k\varphi(a_1\otimes\cdots \otimes a_ka_{k+1}\otimes\cdots \otimes a_{i+1})\]for $\varphi\in (A_+^{\otimes i})^*$
(cf.~\cite[\S~1.1]{PoliPosi}, where the differentials are defined with the opposite sign).
For $\varphi\in Cob^i(A_\bullet)$ and $\psi\in Cob^j(A_\bullet)$,
we can define the tensor product $\varphi\otimes\psi\in Cob^{i+j}(A_\bullet)$ as the map
\begin{equation}
 (\varphi\otimes\psi)(a_1\otimes\dots\otimes a_{i+j})=\varphi(a_1\otimes\dots\otimes a_i)\psi(a_{i+1}\otimes\dots\otimes a_{i+j}).
\end{equation}
This tensor product satisfies
\(
d(\varphi\otimes\psi)=d\varphi\otimes\psi + (-1)^i\varphi\otimes d\psi
\)
and so induces a well-defined product on cohomology classes:
\[\cup\colon\ext_{A_\bullet}^i(\Bbbk,\Bbbk)\otimes\ext_{A_\bullet}^j(\Bbbk,\Bbbk) \to \ext_{A_\bullet}^{i+j}(\Bbbk,\Bbbk)\]
given by $[\varphi]\cup[\psi] = [\varphi\otimes\psi]$.
The map $\cup$ is a special case of Yoneda product called \emph{cup product} and
turns $\ext_{A_\bullet}^\bullet(\Bbbk,\Bbbk)$ into a connected graded $\Bbbk$-algebra. 

Thanks to the finite-type assumption, the grading on $A_\bullet$ induces an additional grading
on $H^\bullet(A_\bullet,\Bbbk)$, as follows.
The isomorphisms 
\begin{equation}
\begin{split}
  &\rho_{i,j} \colon A_i^*\otimes A_j^*\longrightarrow (A_i\otimes A_j)^*\\
 &\rho_{i,j}(f\otimes g) \colon (a\otimes b)\longmapsto f(a)g(b)
\end{split}
\end{equation}
define a map $\Delta=\rho_{\bullet,\bullet}^{-1}\circ\mu^*:A_+^*\to A_+^*\otimes A_+^*$ that is dual to the algebra multiplication $\mu:A_+\otimes A_+\to A_+$.

Then each $Cob^{i}(A_\bullet)$ decomposes as the direct sum of the submodules
\[\begin{split}
   Cob^{i,j}({A_\bullet})\cong(A_+^{\otimes i})^*_j 
   &= \left(\bigoplus_{\substack{k_1+\dots+k_i=j,\\ \text{all } k_s\geq 1}} A_{k_1}\otimes\dots\otimes A_{k_i}\right)^*\\
   &\cong \bigoplus_{\substack{k_1+\dots+k_i=j,\\ \text{all } k_s\geq 1}} \left(A_{k_1}^*\otimes\dots\otimes A_{k_i}^*\right)
  \end{split}
\]
(the condition that all $k_s\geq 1$ is required to stay in $A_+$).
In this description of the cobar complex, the differential reads 
\[
\begin{array}{rcl}
d:\bigoplus_{k_1+\dots+k_i=j,\, \forall k_s\geq 1}\!\left(\!A_{k_1}^*\!\otimes\!\dots\!\otimes\! A_{k_i}^*\!\right)\!\!&\!\!\!\to\!\!&\!\!\! \bigoplus_{k_1+\dots+k_{i+1}=j,\, \forall k_s\geq 1}\!\left(\!A_{k_1}^*\!\otimes\!\dots\!\otimes\! A_{k_{i+1}}^*\!\right)\\
d(\varphi_1\otimes\dots\otimes\varphi_i)\!\!&\!\!=\!\!&\!\!\sum_{k=1}^i(-1)^k\varphi_1\otimes\dots\otimes\Delta\varphi_k\otimes\dots\otimes\varphi_i.
\end{array}
\]
The differential respects the second grading, i.e. $d(\mathcal{C}^{i,j})\subseteq\mathcal{C}^{i+1,j}$,
so this grading passes to cohomology. The cup product is compatible with both gradings:
\[
\cup:\ext_{A_\bullet}^{i,j}(\Bbbk,\Bbbk)\otimes\ext_{A_\bullet}^{r,s}(\Bbbk,\Bbbk)  \to  \ext_{A_\bullet}^{i+r,j+s}(\Bbbk,\Bbbk).
\]
Hence $\ext_{A_\bullet}^{\bullet,\bullet}(\Bbbk,\Bbbk)$ becomes a \emph{bigraded algebra}.
The first grading is called the \emph{homological} grading;
the second one, induced by $A_\bullet$, is called the \emph{internal} one.


\subsection{Koszul property}\label{ssec:koszul prop}
\begin{defn}\label{def:koszul algebra}
A quadratic algebra $A_\bullet$ is a \emph{Koszul algebra} if the trivial $A_\bullet$-module $\Bbbk$
has a graded projective resolution $(P_\bullet, d_\bullet)$ with each $P_i$ generated by its component of degree $i$: $P_i=A_\bullet\cdot(P_i)_i$ (we require all the differentials $d_i$ to have degree $0$; gradings on the $P_i$s are set accordingly).
\end{defn}
If, as in our case, the $\ext$ algebra can be equipped with a well-defined internal grading
(see \S~\ref{ssec:graded}), this is equivalent to $\ext_{A_\bullet}^{\bullet,\bullet}(\Bbbk,\Bbbk)$ being concentrated on the diagonal:
\[
\ext^{i,j}_{A_\bullet}(\Bbbk,\Bbbk)= 0\mbox{ for }i\neq j.
\]

\begin{ex}\label{ex:koszul algebras}
\begin{enumerate}
\item[(a)] Tensor algebras, trivial algebras, symmetric algebras and exterior algebras are Koszul (cf.~\cite[\S~3.4.5]{lodval}).
\item[(b)] Each of the direct sum, the free product, the symmetric tensor product and the skew-symmetric tensor product of two algebras $A_\bullet$ and $B_\bullet$ is Koszul if and only if both $A_\bullet$ and $B_\bullet$ are Koszul (cf.~\cite[\S~3.1]{PoliPosi}).
\end{enumerate}
\end{ex}

The most interesting feature of Koszul algebras is that their cohomology admits a very explicit description.

\begin{thm}[Priddy, \cite{priddy}]\label{thm:koszul cohomology}
An algebra $A_\bullet$ is Koszul if and only if there is a (degree-$0$) isomorphism of graded algebras 
\[H^\bullet(A_\bullet,\Bbbk)=\ext^{\bullet}_{A_\bullet}(\Bbbk,\Bbbk)\cong A_\bullet^!\]
\end{thm}
Also, a quadratic $\Bbbk$-algebra $A_\bullet$ is Koszul if and only if its quadratic dual $A_\bullet^!$
is Koszul (cf.~\cite[Corollary~2.3.3]{PoliPosi}).

For further details we refer to \cite{begiso}, \cite[Chapter 3]{lodval}, \cite[Chapters 1-2]{PoliPosi},
\cite[\S~2]{positselsky:koszul}, \cite{priddy}.


\subsection{PBW property}\label{ssec:PBW prop}
Checking Koszulity with the definition is still a difficult task. PBW property is a strictly stronger condition, but it is generally easier to check.

Let $A_\bullet=Q(V,\Omega)$ be a quadratic algebra.
Suppose a totally ordered basis $\{x_1,\dots,x_d\}$ of $V$ is given. This datum is equivalent to a total order $\preceq_1$ on the set $I=\{1,\dots,d\}$. The set $\mathcal{I}=\sqcup_{n\in\N}I^n$ is the set of multiindices that uniquely identify each monomial in $T_\bullet(V)$ (putting $I^0=\{0\}$ and using it as the singleton set identifying the empty monomial $1$). The concatenation of multiindices $(i_1,\dots,i_k)*(j_1,\dots,j_h)=(i_1,\dots,i_k,j_1,\dots,j_h)$ turns $\mathcal{I}$ into a monoid with $0$ as neutral element.

\begin{defn}\label{def:admissible order}
Let $V=\bigoplus_{i\in I}\gen{x_i}$, $\mathcal{I}$ and $\star$ be as before. An \emph{admissible order} (sometimes called \emph{monoid order}) on $\mathcal{I}$ is a total order $\preceq$ such that
\begin{itemize}
\item[(i)] $\preceq$ extends $\preceq_1$, that is, the two orders coincide on $I$;
\item[(ii)] for all $s\in \mathcal{I}$, $0\preceq s$;
\item[(iii)] for all $r,s,t\in \mathcal{I}$, $r\prec s\Rightarrow (r\star t\prec s\star t\mbox{ and }t\star r\prec t\star s)$.
\end{itemize}
\end{defn}

\begin{ex}\label{ex:deglex}
Given any total order $\preceq_1$ on $I$, the associated degree-lexicographic order
$\preceq_{\mathrm{deglex}}$ on $\mathcal{I}$ is the admissible order defined as follows.
For $s=(i_1,\dots,i_m)\in I^m$ and $t=(j_1,\dots,j_n)\in I^n$, $s\preceq t$
if and only if either $m<n$ or $m=n$ and there exists $k\in \{1,\dots,n\}$ such that
$i_1=j_1,i_2=j_2,\dots,i_{k-1}=j_{k-1}$, and $i_k\preceq_1 j_k$. 
\end{ex}

If one accepts the axiom of countable choice, an admissible order $\preceq$
gives an order isomorphism $index_\preceq:\N\to\mathcal{I}$ such that $index_\preceq(0)=0$. For instance, $index_{\preceq_{\mathrm{deglex}}}(0)=0$, $index_{\preceq_{\mathrm{deglex}}}(1)=1$, \dots, $index_{\preceq_{\mathrm{deglex}}}(d)=d$,
$index_{\preceq_{\mathrm{deglex}}}(d+1)=(1,1)$, \dots.

We obtain an increasing and exhaustive filtration $F_\bullet T_\bullet(V)$ on $T_\bullet(V)$:
\[
F_nT_\bullet(V)=\bigoplus_{\substack{s\in \mathcal{I}\\ s\preceq index_{\preceq}(n)}}\gen{\underline{x}_s},
\]
where we use the multiindex notation $\underline{x}_s$ to denote the monomial identified by $s$. This in turn produces an increasing and exhaustive filtration $F_\bullet A$ on $A_\bullet$, with $F_nA$ given by the image $F_nT_\bullet(V)$ under the canonical projection $T_\bullet(V)\to A_\bullet$. The filtration on $A_\bullet$ gives rise to a weight-graded module 
\[
\gr_\bullet A=\bigoplus_{n\in\N}F_{n+1}A_\bullet/F_{n}A_\bullet,
\]
which depends on the admissible order $\preceq$, and which is a graded algebra with the product and the grading induced from $A_\bullet$. We call it the \emph{weight-graded algebra} of $A_\bullet$ induced by $\preceq$.

The product on $\gr_\bullet A$ is usually simpler than that on $A_\bullet$, nevertheless, a spectral sequence argument shows that if the former is Koszul, the latter is Koszul as well (cf.~\cite[Proposition~4.2.1]{lodval}, \cite[\S~4.7]{PoliPosi} or \cite[\S~5]{priddy}).

The monomial identified by the $\preceq$-maximal multiindex within a given relator $r\in \Omega$ is called
the \emph{leading monomial} of $r$ and denoted $lm(r)$.
Letting $\Omega_{lead}$ be the vector subspace generated by the leading monomials of all elements of $\Omega$,
a tentative presentation of $\gr_\bullet A$ is $T_\bullet(V)/(\Omega_{lead})$. In fact, there is an epimorphism
of graded algebras $\rho:T_\bullet(V)/(\Omega_{lead})\twoheadrightarrow\gr_\bullet A$ which is an isomorphism in degrees $0,1,2$.
The Diamond Lemma (cf.~the original reference \cite{bergman}, or \cite[\S~4.2.4, \S~4.3.5]{lodval})
asserts that the injectivity of $\rho$ in degree $3$ and the Koszulity of $T_\bullet(V)/(\Omega_{lead})$ are sufficient to get the isomorphism $T_\bullet(V)/(\Omega_{lead})\cong\gr_\bullet A$, and hence the Koszulity of $A_\bullet$. On the other side, $T_\bullet(V)/(\Omega_{lead})$ is a \emph{monomial} quadratic algebra, that is, the space of relators is linearly spanned by monomials of degree $2$. Such algebras are always Koszul, for example as a consequence of Backelin's Criterion, cf.
~\cite[\S~2.4]{PoliPosi}.

The injectivity of $\rho$ in degree $3$ can be algorithmically checked with the Rewriting Method
(cf.~\cite[\S~4.1]{lodval}).

First, one chooses a basis $B=\{r_1,\dots,r_m\}$ of $\Omega$. By elementary linear algebra over a field,
one can always suppose that the coefficients of $lm(r_1),\dots,lm(r_m)$ are $1$ and that,
for each $i$, $lm(r_i)$ does not appear as a monomial in any other relator $r_j$, $j\neq i$. A basis with these additional properties is said to be \emph{normalized}.

Each relation $r_i=0$ can be interpreted as a rewriting rule
\[
lm(r_i)\rightsquigarrow-\mbox{(sum of non maximal monomials)}.
\]
Rewriting rules can be subsequently applied to an initial monomial until the result does not contain
any more leading terms. This final shape is usually called a \emph{normal form} of the initial monomial.
The process can be depicted with a directed graph, whose vertices are the polynomials obtained as steps of the aforementioned process
and whose edges connect a polynomial to a second polynomial that is obtained applying
a single rewriting rule to the first.
Normal forms are then the terminal vertices of these \emph{rewriting graphs}.

A degree $3$ monomial $x_ax_bx_c$ is \emph{critical} if both $x_ax_b$ and $x_bx_c$ are leading terms of relators in $B$. This implies that there are two possible rewriting processes on it, starting with either the rewriting rule on $x_ax_b$ or that on $x_bx_c$. The resulting rewriting graph may then have two distinct terminal vertices, that is, the critical monomial may have two distinct normal forms. A critical monomial is \emph{confluent} if its rewriting graph has only one terminal vertex.

The (cosets of the) terminal vertices of the rewriting graph of a critical monomial are equal in $\gr_\bullet A$, but distinct in $T_\bullet(V)/(\Omega_{lead})$. Since this is the only obstruction to the injectivity of $\rho$ in degree $3$, this injectivity is equivalent to the confluence of every critical monomial. $A_\bullet$ is called a \emph{PBW algebra} if there exist bases for $V$ and $\Omega$ and an admissible order as above with respect to which every critical monomial is confluent. In that case, the above basis of $V$ is called a set of \emph{PBW generators} of $A_\bullet$.

The whole discussion can be then summarised in the fact that a (quadratic) PBW algebra is Koszul.

\begin{rem}\label{rem:PBW}
PBW property is preserved under the quadratic dual construction.
More precisely, if $\{x_1,\dots,x_n\}$ is a set of PBW generators of the quadratic algebra $A_\bullet$
under the suitable order $\preceq$, then the dual basis $\{x_1^*,\dots,x_n^*\}$ is a set
of PBW generators of $A_\bullet^!$ under the \emph{opposite} suitable order:
$s\preceq^{op} t \Leftrightarrow t\preceq s$ (cf.~\cite[\S~4.3.9]{lodval}).
\end{rem}

\begin{ex}\label{ex:PBW}The following are PBW algebras that will be relevant subsequently.
\begin{enumerate}
\item[(a)] A quadratic monomial algebra is a quadratic algebra admitting a presentation in which all relators are (noncommutative) monomials of degree $2$ in the generators. All quadratic monomial algebras are PBW. Indeed, consider a set of algebra generators with respect to which all relators are quadratic monomials. Then all the associated critical monomials reduce to $0$ after one step in either way the rewriting rules are applied. In particular, tensor algebras and trivial algebras are PBW.
\item[(b)] Symmetric and exterior algebras are PBW under the degree-lexicographic order.
\item[(c)] Let $A_\bullet=Q(V_A,\Omega_A)$ and $B_\bullet=Q(V_B,\Omega_B)$ be two quadratic PBW algebras. Then the quadratic algebras $A_\bullet\sqcap B_\bullet$, $A_\bullet\sqcup B_\bullet$, $A_\bullet\otimes^1 B_\bullet$ and $A_\bullet\otimes^{-1}B_\bullet$ introduced in Example \ref{ex:quadratic2} are PBW (cf.~\cite[\S~4.4]{PoliPosi}). In each case, a set of PBW generators is the union of the sets of PBW generators of $A_\bullet$ and $B_\bullet$.
\end{enumerate}

\end{ex}

\begin{lem}\label{lem:one-relator PBW}
Let $A=\Bbbk\gen{x_1,\dots,x_d\mid r}=Q(\gen{x_1,\dots,x_d}, \gen{r})$ a quadratic algebra presented with a single relator $r$. Suppose that there is an index $i\in I=\{1,\dots,d\}$ such that $x_i^2$ is not a monomial of $r$ but some $x_ix_j$, $j\neq i$, is. Then $A$ is PBW.
\end{lem}
\begin{proof}
We can choose any total order $\preceq_1$ on $I$ with maximal element $i$ and second-to-maximal element $j$. As admissible order on $\mathcal{I}$, we use the associated degree-lexicographic order $\preceq_{\mathrm{deglex}}$ introduced in Example \ref{ex:deglex}.
Then the leading monomial $lm(r)$ is $x_ix_j$, and since $j\neq i$, there are no critical monomials at all.
\end{proof}


\section{Filtrations on groups and associated algebras}
\label{sec:zass}
In this section we define a filtration on an arbitrary group $G$. It provides a stratification of $G$ into layers
in which the group product has a simpler description.
The resulting graded object has a natural structure of a restricted Lie algebra (see Definition~\ref{defi:restricted}).
This structure is intimately related to the ring-theoretic structure of the group algebra of $G$ over $\F_p$.
Since group algebras are the basic ingredients from which group cohomology is defined, the restricted Lie algebra
of a pro-$p$ group is expected to play a significant role in Galois cohomology.
This will be addressed in the following sections.

Even though the contents of this section applies to any group, we are mostly interested in pro-$p$ or profinite groups.
Exercise (6.d) in \cite[Ch.~I, \S~4.2]{serre} asks to prove that the topology of a finitely generated pro-$p$ group
is determined by the group structure: more precisely, any subgroup of finite index is open.
A full proof of this can be found in \cite[Theorem~1.17]{ddsms:analytic}.
The much more difficult statement that the same is true
for every finitely generated profinite group was proved in \cite{nikseg1, nikseg2}.
However, it is not true that every subgroup of a finitely generated profinite (or pro-$p$) group is closed.
We shall therefore adopt the usual
\begin{conv}\rm \
{When dealing with profinite (or pro-$p$) groups, every generating set is to be intended in the topological sense, every subgroup is understood to be closed and every map is understood to be continuous}.
\end{conv}

\subsection{Filtrations}
\label{ssec:zass}
Given two subgroups $C_1$ and $C_2$ of $G$, $[C_1,C_2]$ is the subgroup of $G$ generated by the commutators
\[ [x,y] = (y^{-1})^x\cdot y = x^{-1}y^{-1}xy , \qquad x\in C_1 , y\in C_2 .\]
Moreover, for $n\geq1$, $C_1^n$ is the subgroup of $G$ generated by the elements $g^n$, $g\in C_1$.

\begin{defn}\label{defi:zass}
Let $G$ be a group.
\begin{itemize}
\item[(i)] The \emph{descending central series} $\{\gamma_i(G)\}$ of $G$ is defined by
\[\begin{array}{lclr}
\gamma_1(G)&=&G& \\
\gamma_{n}(G)&=&[G,\gamma_{n-1}(G)], & n\geq2.
\end{array}
\]
\item[(ii)]
The \emph{$p$-Zassenhaus filtration} (or simply \emph{Zassenhaus filtration}, if the prime is clear from the context) of $G$ is 
defined by
\begin{equation}\label{zassenhaus defi}
\begin{array}{lclr}
G_{(1)}&=&G& \\
 G_{(n)}&=&G_{(\lceil n/p\rceil)}^p\cdot\prod_{i+j=n}[G_{(i)},G_{(j)}], & n\geq2.
\end{array}
\end{equation}
Here $\lceil n/p\rceil$ is the least integer $h$ such that $hp\geq n$.
\end{itemize}
\end{defn}
Observe that $G_{(2)}=G^p[G,G]$, which, in the case of $G$ being a pro-$p$ group, equals the Frattini subgroup $\Phi(G)$.
\begin{rem}
The subgroups $G_{(i)}$ are often called the \emph{modular dimension subgroups} of $G$, for example in \cite[Chapter 11]{ddsms:analytic}. Accordingly, the Zassenhaus filtration is also known as the \emph{dimension series}.
\end{rem}

The Zassenhaus filtration of $G$ is the fastest descending series starting at $G$ satisfying
\[ [G_{(i)},G_{(j)}]\subseteq G_{(i+j)}\quad\text{and}\quad G_{(i)}^p\subseteq G_{(ip)}\]
for every $i,j\geq 1$.
Moreover, every quotient $G_{(n)}/G_{(n+1)}$ is a $p$-elementary abelian group, and thus a vector space over $\F_p$.
M.~Lazard established the useful formula
\begin{equation}\label{eq:Lazard}
 G_{(n)}=\prod_{ip^h\geq n}\gamma_i(G)^{p^h}, \qquad \text{for all }n\geq1
\end{equation}
(cf.~\cite[Theorem~11.2]{ddsms:analytic}).
In particular, 
\begin{equation}\label{eq:G3}
 G_{(3)}=\begin{cases} G^p[G,[G,G]] & \text{if }p\neq2 \\ G^4[G,G]^2[G,[G,G]] & \text{if }p=2.\end{cases}
\end{equation}
%

\subsection{Group algebras}
Let $G$ be a group and let $\F_p[G]$ be its group algebra over the finite field $\F_p$.
The \emph{augmentation ideal} $I_G$ of $\FpG$ is the kernel of the augmentation map $\epsilon\colon\FpG\to\F_p$,
given by $g\mapsto1$ for every $g\in G$. The augmentation ideal induces a filtration $F^n\FpG={I_G}^n$, and defines an associated graded object 
\[\gr\FpG=\bigoplus_{n\geq0}{I_G}^{n}/ {I_G}^{n+1},\] 
called the \emph{graded group $\F_p$-algebra} of $G$.
To avoid confusion with the previous notation for subgroups, we stress that ${I_G}^0=\FpG$ and that ${I_G}^n$ denotes the usual $n$-fold ideal product of ${I_G}$.

If $G$ is a pro-$p$ group, there are completed versions of all the previous concepts (cf.~\cite[\S~7.1]{koch}).
\begin{defn}\label{defi:FpG}
The \emph{complete group algebra} of a pro-$p$ group $G$ is 
\[
\FppG=\varprojlim_{U}\F_p[G/U]
\]
where $U$ runs through all open normal subgroups of $G$.
The algebra $\FppG$ is a topological, compact $\F_p$-algebra.
The closed two-sided ideal $\widehat{I_G}$ generated by the closure of the kerne
of the augmentation map $\epsilon\colon\FppG\to\F_p$ is open and it is called the \emph{complete augmentation ideal}.

The complete augmentation ideal induces a filtration $F^n\FppG=\widehat{I_G}\!\ ^n$, and defines an associated graded object 
\[\gr\FppG=\bigoplus_{n\geq0}\widehat{I_G}\!\ ^{n}/ \widehat{I_G}\!\ ^{n+1},\] 
called the \emph{complete graded group $\F_p$-algebra} of $G$.
Here $\widehat{I_G}\!\ ^0=\FppG$ and $\widehat{I_G}\!\ ^n$ is the usual $n$-fold ideal product of $\widehat{I_G}$.
\end{defn}

The Zassenhaus filtration and the graded group $\F_p$-algebra of a group are related by the following fundamental result due
to S.~A.~Jennings (cf.~\cite{jennings}, \cite[\S~12.2]{ddsms:analytic}).
\begin{prop}\label{prop:zassenhaus grFpG}
Let $G$ be a group. Then $G_{(n)}=\{g\in G\mid g-1\in I_G^n\}$ for all $n\geq1$.
\end{prop}
Again, if $G$ is a pro-$p$ group, a completed version of Jennings's Theorem holds (cf.~\cite[\S~7.4]{koch}):
\begin{prop}\label{prop:complete zassenhaus grFpG}
Let $G$ be a pro-$p$ group. Then $G_{(n)}=\{g\in G\mid g-1\in \widehat{I_G}\!\ ^n\}$ for all $n\geq1$.
\end{prop}


\subsection{Restricted Lie algebras}
\label{ssec:restricted}

The following notion was introduced by N.~Jacobson (cf.~\cite[\S~V.5]{jacobson}).

\begin{defn}\label{defi:restricted}
A \emph{restricted Lie algebra} $L$ over $\F_p$ is a Lie $\F_p$-algebra equipped with
an additional unary operation $\slot^{[p]}\colon L\to L$, called $p$-operation, satisfying
\begin{itemize}
 \item[(i)] $(\alpha a)^{[p]}=\alpha^pa^{[p]}$ for all $\alpha\in\F_p, a\in L$;
 \item[(ii)] $(a+b)^{[p]}=a^{[p]} + \sum_{i=1}^{p-1} i^{-1}s_i(a,b) + b^{[p]}$ for all $a,b\in L$, where $s_i(a,b)$ is the coefficient of $T^{i-1}$ in $\mathrm{ad}(Ta+b)^{p-1}(a)$;
 \item[(iii)] $\mathrm{ad}(a^{[p]})=\mathrm{ad}(a)^p$ for all $a\in L$.
\end{itemize}
We recall that the map $\mathrm{ad}:L\to \mathrm{End}(L)$ is defined as $\mathrm{ad}(a)=\mathrm{ad}_a:b\mapsto [a,b]$.

A morphism of restricted Lie algebras is a morphism of Lie algebras that preserves the $p$-operations.
\end{defn}

For any associative $\F_p$-algebra $A$, the bracket operation $[a,b] := ab-ba$ and the $p$-operation $a^{[p]} := a^p$ make the $\F_p$-module underlying $A$ into a restricted Lie algebra $A_L$. This defines a functor from associative algebras to restricted Lie algebras, which has a left adjoint functor $\sU$. For a restricted Lie algebra $L$, $\sU(L)$ is called the \emph{universal restricted enveloping algebra} of $L$. Concretely, $\sU(L)$ is the quotient of the tensor algebra $T(L)$ by the two-sided ideal generated by the elements $ab-ba-[a,b]$ and $a^{[p]}-a^p$, $a,b\in L$. It comes equipped with an injective function, the \emph{universal embedding} $\vartheta:L\to\sU(L)$, that is a monomorphism of restricted Lie algebras from $L$ to $\sU(L)_L$ (cf.~\cite[\S~V.5]{jacobson}). This makes precise the idea that the $p$-operation is the Lie analogue of the $p$\textsuperscript{th} power map in the associative framework.

A \emph{restricted Lie subalgebra} $M$, respectively a \emph{restricted ideal} $\mathfrak{r}$, of a restricted Lie $\F_p$-algebra $L$ is a Lie subalgebra, respectively an ideal, of $L$ as Lie algebra with the further condition that $a^{[p]}\in M$ for every $a\in M$, respectively $a^{[p]}\in\mathfrak{r}$ for every $a\in\mathfrak{r}$.
As remarked in \cite[Proposition~2.1]{jochen:massey}, the well-known relationship between quotients of Lie algebras and universal enveloping algebras (cf.~\cite[\S I.2.3, Proposition~3]{bourbaki:lie}) remains true in the restricted framework:
\begin{prop}\label{prop:presentation U}
Let $L$ be a restricted Lie algebra over $\F_p$ and let $\mathfrak{r}\subseteq L$ be a restricted ideal.
Let $\mathcal{R}$ be the left ideal of $\sU(L)$ generated by the image of $\mathfrak{r}$
via the universal embedding $\vartheta\colon L\to\sU(L)$.
Then $\mathcal{R}$ is a two-sided ideal and the epimorphism $L\to L/\mathfrak{r}$ induces a short exact sequence
\begin{equation}\label{eq:ses U}
 \xymatrix{ 0\ar[r] & \mathcal{R}\ar[r] & \sU(L)\ar[r] & \sU(L/\mathfrak{r})\ar[r] & 0. }
\end{equation}
\end{prop}

The \emph{free product} of two restricted Lie $\F_p$-algebras $L_1$ and $L_2$ is the coproduct of $L_1$ and $L_2$ in the category of restricted Lie $F_p$-algebras. That is, $L_1\ast L_2$ contains (isomorphic images of) $L_1$ and $L_2$ as restricted Lie subalgebras, and for any couple of morphisms of restricted Lie algebras, $\psi_1\colon L_1\to H$ and $\psi_2\colon L_2\to H$, there exists a unique morphism $\psi \colon L_1\ast L_2\to H$ such that $\psi_i=\psi\vert_{L_i}$ for $i=1,2$ (cf.~\cite[Remark 1]{lichtman:lie}). For the existence of the free product see \cite[Lemma 2]{lichtman:lie}.

The \emph{direct sum} of two restricted Lie $\F_p$-algebras $L_1$ and $L_2$ is the quotient of $L_1\ast L_2$
by the restricted ideal generated by the Lie brackets $[a,b]$ with $a\in L_1$ and $b\in L_2$.
\begin{rem}\label{rem:construction general algebras}
The free product and the symmetric tensor product constructions, recalled in Example \ref{ex:quadratic2} for quadratic algebras, make sense for general associative algebras. If $A$ and $B$ are algebras over a commutative ring $R$, the free product $A\sqcup B$ is the coproduct of $A$ and $B$ in the category of (associative, unital) $R$-algebras; an explicit description can be found in \cite[\S~1.4]{bemami}. The symmetric tensor product $A\otimes^1 B$ is the quotient of $A\sqcup B$ by the ideal generated by the elements $ab - ba$ with $a\in A$ and $b\in B$.
\end{rem}
\begin{prop}\label{prop:U prod sum}
 Let $L_1$ and $L_2$ be two restricted Lie algebras over $\F_p$.
\begin{itemize}
 \item[(i)] The universal restricted enveloping algebra of the free product $L_1\ast L_2$ is 
the free product $\sU(L_1)\sqcup\sU(L_2)$.
 \item[(ii)] The universal restricted enveloping algebra of the direct sum $L_1\oplus L_2$ is the symmetric tensor product $\sU(L_1)\otimes^1\sU(L_2)$.
\end{itemize}
\end{prop}
\begin{proof}
\begin{itemize}
\item[(i)] This is asserted in \cite[Lemma 2]{lichtman:lie}, but there the proof is left to the reader. For the sake of completeness, we provide a proof.

First, we define a candidate universal embedding
\[\vartheta:L_1\ast L_2\to\sU(L_1)\sqcup\sU(L_2).\]
Define the inclusion maps $\eta_1:\sU(L_1)\to\sU(L_1)\sqcup\sU(L_2)$, $\eta_2:\sU(L_1)\to\sU(L_1)\sqcup\sU(L_2)$
and $\iota_1:L_1\to L_1\ast L_2$, $\iota_2:L_2\to L_1\ast L_2$.
There are universal embeddings $\vartheta_1:L_1\to\sU(L_1)$,
$\vartheta_2:L_2\to\sU(L_2)$ which, composed with $\eta_1,\eta_2$, give two morphisms of Lie algebras
$L_1\rightarrow (\sU(L_1)\sqcup\sU(L_2))_L \leftarrow L_2$. 
The universal property of $L_1\ast L_2$
as a coproduct gives a unique restricted Lie algebra morphism $\vartheta:L_1\ast L_2\to(\sU(L_1)\sqcup\sU(L_2))_L$
with the property that $\vartheta\iota_1=\eta_1\vartheta_1,\;\vartheta\iota_2=\eta_2\vartheta_2$.
\[
\xymatrix@R-8pt@C-8pt{
L_1\ar^{\iota_1}[rr]\ar^{\vartheta_1}[rd] && L_1\ast L_2\ar@{-->}_{\exists!}^{\vartheta}[dd] && L_2\ar_{\iota_2}[ll]\ar_{\vartheta_2}[ld] \\
& \sU(L_1)_L\ar^{(\eta_1)_L=\eta_1}[dr] && \sU(L_2)_L\ar_{(\eta_2)_L=\eta_2}[dl] \\
&& (\sU(L_1)\sqcup\sU(L_2))_L &&
}
\]

Then we prove that for any associative algebra $M$ and any restricted Lie morphism $\varphi:L_1\ast L_2\to M$ there is a unique algebra morphism $\alpha:\sU(L_1)\sqcup\sU(L_2)\to M$ such that $\alpha\vartheta=\varphi$.
Let $\varphi_1=\varphi\iota_1,\;\varphi_2=\varphi\iota_2$. By the universal property of universal restricted enveloping algebras, there are unique algebra morphisms $\alpha_1,\alpha_2$ such that $\alpha_1\vartheta_1=\varphi_1,\;\alpha_2\vartheta_2=\varphi_2$.
Since $\sU(L_1)\sqcup\sU(L_2)$, together with $\eta_1,\eta_2$, is a coproduct of unital, associative $\F_p$-algebras, there is a unique algebra morphism $\alpha:\sU(L_1)\sqcup\sU(L_2)\to M$ satisfying $\alpha\eta_1=\alpha_1,\; \alpha\eta_2=\alpha_2$.
\[
\xymatrix{
&& L_1\ast L_2\ar^{\vartheta}[d] && \\
L_1\ar^{\iota_1}[urr]\ar_{\vartheta_1}[r]\ar_{\varphi_1}[ddrr] & \sU(L_1)\ar_-{\eta_1}[r]\ar^{\alpha_1}[rdd] & \sU(L_1)\sqcup\sU(L_2)\ar^{\alpha}[dd] & \sU(L_2)\ar^-{\eta_2}[l]\ar_{\alpha_2}[ldd] & L_2\ar_{\iota_2}[ull]\ar^{\vartheta_2}[l]\ar^{\varphi_2}[ddll] \\
&&&&\\
&& M &&
}
\]
Since $(\alpha\vartheta)\iota_j=\alpha\eta_j\vartheta_j=\alpha_j\vartheta_j=\varphi_j=\varphi\iota_j$ for $j=1,2$, by the universal property of $L_1\ast L_2$ we have $\alpha\vartheta=\varphi$. If also $\beta:\sU(L_1)\sqcup\sU(L_2)\to M$ is an algebra morphism satisfying $\beta\vartheta=\varphi$, then for $j=1,2$
\[
\begin{array}{lcr}
& \beta\eta_j\vartheta_j=\beta\vartheta\iota_j=\varphi\iota_j=\varphi_j & \\
\Rightarrow & \beta\eta_j=\alpha_j & \mbox{(uniqueness of $\alpha_j$)} \\
\Rightarrow & \beta=\alpha & \mbox{(uniqueness of $\alpha$)}.
\end{array}
\]
In particular, taking $M=\sU(L_1\ast L_2)$ and $\varphi$ the universal embedding, we see that $\vartheta$ is injective, as required in the definition of the universal restricted enveloping algebra.
Summing up, $\sU(L_1\ast L_2)\cong\sU(L_1)\sqcup\sU(L_2)$.
\item[(ii)] By the definition of direct sum, Proposition~\ref{prop:presentation U} and (i) imply that \[\sU(L_1\oplus L_2)\cong\frac{\sU(L_1)\sqcup\sU(L_2)}{\mathcal{R}},\] with $\mathcal{R}$ the two-sided ideal generated by the commutators $[\vartheta(v_1),\vartheta(v_2)]=\vartheta(v_1)\vartheta(v_2)-\vartheta(v_2)\vartheta(v_1)$.
\end{itemize}
\end{proof}

For a group $G$, let $L(G)$ be the graded object \[L(G)=\bigoplus_{n\geq1}G_{(n)}/G_{(n+1)}.\]
The group commutator and the $p$-power of $G$ induce the structure of graded restricted Lie algebra on $L(G)$.

The two versions of Jennings's Theorem (Propositions \ref{prop:zassenhaus grFpG} and \ref{prop:complete zassenhaus grFpG})
can be reformulated and completed in terms of $L(G)$ (cf. \cite{quillen} and \cite[Theorem~12.8]{ddsms:analytic}).

\begin{thm}\label{thm:LG grFpG}
Let $G$ be a group, and for an element $g\in G_{(n)}$ set $\overline g=gG_{(n+1)}\in G_{(n)}/G_{(n+1)}$.
Then the assignment
 \[
  \vartheta(\overline g)=(g-1)+I_G^{n+1}\in I_G^n/I_G^{n+1}
 \]
induces a monomorphism of restricted Lie algebras $\vartheta\colon L(G)\to(\gr\FpG)_L$,
such that $\gr\FpG$ endowed with $\vartheta$ is the universal restricted enveloping algebra of $L(G)$.

If $G$ is a pro-$p$ group, the same is true with $\FppG$ and $\widehat{I_G}$ in place of $\FpG$ and $I_G$.
\end{thm}

We call $\overline g\in L(G)$ the \emph{initial form} of $g\in G$. If $\overline g\in G_{(n)}/G_{(n+1)}$, $\overline{g}\neq1$, we say that the inital form of $g$ has degree $n$.

\begin{rem}
Let $G$ be a \emph{profinite} group with finitely generated maximal pro-$p$ quotient $G(p)$.
Since the quotient of $G$ over any term of the $p$-Zassenhaus filtration is a finite $p$-group,
the restricted Lie algebras $L(G(p))$ and $L(G)$ are isomorphic,
and so are their respective restricted universal enveloping algebras $\sU(L(G(p)))$ and $\sU(L(G))$.
Therefore, Theorem~B also admits an extension to absolute Galois groups $G_F$ with maximal
pro-$p$ quotient $G_F(p)$ which is of elementary type.
\end{rem}

\subsection{Cohomology of pro-$p$ groups}\label{ssec:cohom of pro-p}
The natural cohomology theory to be used with profinite groups is Galois cohomology. It is defined in complete analogy with classical group cohomology, except for the requirement that all maps be continuous. Good accounts on Galois cohomology can be found for example in \cite{nsw:cohm}, \cite{serre}. We content ourselves to quickly recall some useful facts, with a focus on finitely generated pro-$p$ groups, although the results are essentially valid in a broader context.

If $G$ is a pro-$p$ group, its Frattini subgroup is $\Phi(G)=G^p[G,G]$ and the quotient $G/G_{(2)}=G/\Phi(G)$ is a compact elementary abelian $p$-group. Its Pontryagin dual is $H^1(G,\F_p)$ and the cardinality $\dd(G)$ of a minimal generating set of $G$ equals $\dim_{\F_p}H^1(G,\F_p)$ (cf.~\cite[Ch.~I, \S~4.2]{serre}). Thus, if $G$ is finitely generated, $G/\Phi(G)\cong (\Z/p\Z)^{\dd(G)}$. Moreover, if $G$ is a finitely generated pro-$p$ group, and $\{x_1,\dots,x_d\}$ is a minimal generating set of $G$, then the minimal number of defining relations between the $x_i$ is equal to $\dim_{\F_p}H^2(G,\F_p)$ (cf.~\cite[Ch.~I,\S~4.3]{serre}).

For a profinite group $G$, the space $H^\bullet(G,\F_p)=\bigoplus_{n\geq0}H^n(G, \F_p)$ is a connected graded $\F_p$-algebra with respect to the cup product
\begin{equation}\label{eq:cup product}
 \cup\colon H^r(G,\F_p)\times H^s(G,\F_p)\longrightarrow H^{r+s}(G,\F_p),\qquad r,s\geq0
\end{equation}
(cf.~\cite[Proposition~1.4.4]{nsw:cohm}).
The cup product is graded-commutative, that is
\[
a\cup b=(-1)^{rs}b\cup a \qquad\mbox{for }a\in H^r(G,\F_p),b\in H^s(G,\F_p).
\]

This is of special interest when $G$ is a Galois group.
If $\K$ contains a primitive $p$\textsuperscript{th} root of unity, the Kummer map
$\kappa:\K^\times\to H^1(G_\K,\F_p)$ induces an isomorphism $\K^\times/(\K^\times)^p\to H^1(G_\K,\F_p)$,
thereby giving a concrete description of the first cohomology group of $G_\K$.
The possibility to extend this to higher cohomology groups was first suggested by J.~Milnor for $p=2$
(cf.~\cite{milnor}) and then conjectured by S.~Bloch and K.~Kato for any prime (cf.~\cite{BK}).
In detail, the \emph{Milnor K-theory} of $\K$ is the quotient $K^M_\bullet(\K)$ of the tensor algebra (see \S 1.1)
of the multiplicative group $\K^\times$ by the two-sided ideal generated by the tensors $a\otimes(1-a)$,
$a\in\K\setminus\{0,1\}$. The tensor powers $\kappa\otimes\dots\otimes\kappa$ factor through the aforementioned ideal
(Steinberg relations, cf.~\cite[Proposition 4.6.1]{gilsza} or \cite[proof of Theorem 3.1]{tate:rel}) and $H^\bullet(G_\K,\F_p)$ is a $p$-torsion group.
Hence for all $n\geq 1$ there is a well-defined \emph{norm residue map} 
\[
K^M_n(\K)/pK^M_n(\K)\to H^n(G_\K,\F_p).
\]
The Bloch-Kato conjecture claims that all these maps are group isomorphisms.
Relatively recently, M.~Rost and V.~Voevodsky completed the proof of the full Bloch-Kato conjecture, obtaining a very good insight into the structure of Galois cohomology (cf.~\cite{suslin:norm, haeseweibel, voevodsky:motivic, voevodsky:BK}).
A relevant consequence of their achievement is that, if $\K$ contains a primitive $p$\textsuperscript{th} root of unity, then the algebra $H^\bullet(G_{\K},\F_p)$ 
is quadratic.
From this and the graded-commutativity of the cup product, it follows that there is an epimorphism of quadratic $\F_p$-algebras $\Lambda_\bullet(H^1(G,\F_p))\to H^\bullet(G,\F_p)$. In particular, for $p=2$ this is the same as an epimorphism of quadratic algebras $\mathcal{S}_\bullet(H^1(G,\F_p))\to H^\bullet(G,\F_p)$.

\section{Elementary type pro-$p$ groups}\label{sec:ET}
Few groups are known to be realizable as maximal pro-$p$ quotients of absolute Galois groups of fields satisfying Hypothesis \ref{ass} (in short: \emph{realizable}). At present, all the finitely generated groups that are known to be realizable fall into a restricted class.
This class was introduced by I.~Efrat (cf. \cite[\S~3]{efrat:small}), and it
is defined inductively, using two basic types of groups and two operations.
In this section we recall the definitions of these basic groups and operations, starting from a preliminary concept. 

\begin{defn}\label{defn:cyclochar}
We denote the multiplicative group of units of $\Z_p$ by $U_p$.
A \emph{cyclotomic pair} is a couple $(G,\chi)$ made of a pro-$p$ group $G$ and a continuous homomorphism $\chi\colon G\to U_p$ (see \cite{efrat:cycpair}, and \cite{jacwar}, \cite{jacwar2} for the case $p=2$). We call $\chi$ an \emph{orientation} of $G$
(following the notation used in \cite{cq:bk,cmq:fast,qw:cyc}).
\end{defn}

\begin{rem}\label{rem:character pro-p image}
Note that $U_p$ decomposes as the direct product of a pro-$p$ part $U_p^{(1)}$ and a cyclic group of order $p-1$.
Moreover, if $p=2$ then $U_2=U_2^{(1)}\simeq\Z/2\Z\oplus\Z_2$.
The image of a continuous homomorphism from a pro-$p$ group to $U_p$ is included in $U_p^{(1)}$ (cf.~\cite[Ch.~I, \S~4.5]{serre}).
\end{rem}

This terminology is justified in view of the situation of a field $\K$ satisfying Hypothesis \ref{ass}.
The action of $G_{\K}(p)$ on the group $\mu_{p^\infty}(\K_{\mathrm{sep}})$ of the roots of unity of order a power of $p$ lying in a separable closure $\K_{\mathrm{sep}}$ induces a homomorphism $\chi_p\colon G\to U_p$. This homomorphism is known in the literature as the ($p$-adic) cyclotomic character.
\begin{defn}
A cyclotomic pair $(G,\chi)$ is \emph{realizable} if there exists a field $\K$ satisfying Hypothesis \ref{ass}
for which $G=G_{\K}(p)$ and $\chi=\chi_p$.
\end{defn}

We extend the terminology relative to group properties to cyclotomic pairs, so that for example a finitely generated cyclotomic pair is a cyclotomic pair $(G,\chi)$ such that $G$ is finitely generated as a pro-$p$ group.

\subsection{Basic groups}\label{ssec:basic groups}

{ Recall that} a finitely generated pro-$p$ group is \emph{free} on a set $I=\{x_1,\dots,x_d\}$ if it is the inverse limit of the groups $L(I)/M$, where $L(I)$ is the discrete free group on $I$ and $M$ is a normal subgroup of $L(I)$ with index a (finite) power of $p$ (cf., e.g., \cite[\S~3.3]{rz:prof}).
Free pro-$p$ groups satisfy the usual universal property of free objects (cf.~\cite[Ch.~I, \S~1.5]{serre}, where not necessarily finitely generated free pro-$p$ groups are introduced).

\begin{defn}
A pro-$p$ group is called a \emph{Demushkin} group if $\dim_{\F_p} H^1(G,\F_p)$ is finite, $H^2(G,\F_p)$ is isomorphic to $\F_p$,
and the cup product induces a non-degenerate (skew-symmetric) pairing \[H^1(G,\F_p)\times H^1(G,\F_p)\stackrel{\cup}{\longrightarrow}H^2(G,\F_p)\cong \F_p.\]
\end{defn}
By the previous discussion, Demushkin groups are finitely generated.
The only finite Demushkin group is the group of order $2$ (cf.~\cite[Proposition~3.9.10]{nsw:cohm}).
It has a unique orientation with image $\{-1,+1\}$.
Infinite Demushkin groups are precisely the Poincar\'e duality groups of dimension $2$ (cf.~\cite[Ch.~I, \S~4.5]{serre}).
Infinite Demushkin pro-$p$ groups are classified by two invariants. The first is the minimal number of generators. As regards the second invariant,
each infinite Demushkin group $G$ has a canonical orientation $\chi_G$,
with the additional properties 1, 2 and 3 in \cite[Proposition 6]{labute:demushkin}
(cf.~ also \cite{serre:demushkin}\footnote{This is a reprint of the original paper contained in \cite{bourbaki}}).
We then define the second invariant $q=q(G)$ to be $2$ if $\mathrm{Im}(\chi_G)=\{-1,+1\}$, and to be the maximal power $p^k$ such that $\mathrm{Im}(\chi_G)\subseteq 1+p^k\Z_p$ otherwise. 

Using these two invariants, J.~Labute completed the classification of Demushkin groups started by S.~Demushkin and extended by  J.-P.~Serre (cf.~\cite{labute:demushkin}, \cite[\S~III.9]{nsw:cohm}, \cite{serre:demushkin}), summarized in the following.

\begin{thm}[Demushkin, Serre, Labute]\label{thm:presentation demushkin}
A finitely generated pro-$p$ group is an infinite Demushkin group if and only if it can be presented by a minimal set of generators $\{x_1,\ldots,x_d\}$, subject to one relator $r$ that has either of the following forms.
\begin{itemize}
 \item[(i)] If $q\neq2$ (possibly $q=0$), then \begin{equation}\label{eq:demu1}r=x_1^{q}[x_1,x_2][x_3,x_4]\cdots[x_{d-1},x_d],\end{equation} with $d\geq 2$ necessarily even.
 \item[(ii)] If $q=2$ and $d\geq 3$ is odd, then \begin{equation}\label{eq:demu2}r=x_1^2x_2^{2^f}[x_2,x_3][x_4,x_5]\ldots[x_{d-1},x_d]\end{equation} with $f\in\{2,3,\ldots\}\cup\{\infty\}$.
 \item[(iii)] If $q=2$, $d\geq 2$ is even and $[\mathrm{Im}(\chi_G) : (\mathrm{Im}(\chi_G))^2]=2$, then \begin{equation}\label{eq:demu3}r=x_1^{2+2^f}[x_1,x_2][x_3,x_4]\cdots[x_{d-1},x_d]\end{equation} with $f\in\{2,3,\ldots\}\cup\{\infty\}$.
 \item[(iv)] If $q=2$, $d\geq 4$ is even and $[\mathrm{Im}(\chi_G) : (\mathrm{Im}(\chi_G))^2]=4$, then \begin{equation}\label{eq:demu4}r=x_1^2[x_1,x_2]x_3^{2^f}[x_3,x_4]\cdots[x_{d-1},x_d]\end{equation} with $f\in\{2,3,\ldots\}$.
\end{itemize}
Here by convention $p^\infty=0$.

Conversely, such presentations define an infinite Demushkin group for each value of the involved parameters.
The last three cases describe pro-$2$ groups.
\end{thm}

\subsection{Operations}\label{ssec:operations}
\begin{defn}\label{def:free prod}
The \emph{free product} $G_1\ast_p G_2$ of two pro-$p$ groups $G_1,G_2$ is the coproduct of $G_1$ and $G_2$ in the category of pro-$p$ groups. Explicitly, let $G$ be the discrete free product of $G_1$ and $G_2$ and let $\mathcal{N}$ be the family of normal subgroups $N$ of $G$ such that $G/N$ is a finite $p$-group and $N\cap G_1, N\cap G_2$ are open subgroups of $G_1, G_2$ respectively. Then
\[G_1\ast_p G_2=\underleftarrow{\lim}_{N\in\mathcal N}G/N\]
(cf.~\cite{binewe}).

The \emph{free product} of two cyclotomic pairs $(G_1,\chi_1)$ and $(G_2, \chi_2)$ is the cyclotomic pair $(G_1\ast_pG_2, \chi_1\ast_p\chi_2)$, where $\chi_1\ast_p\chi_2$ is the orientation induced by $\chi_1$ and $\chi_2$ via the universal property of the coproduct
(using Remark \ref{rem:character pro-p image}, see also \cite[\S~3]{efrat:small} and \cite[\S~3.4]{qw:cyc}).
\end{defn}
For properties of the free product in the category of pro-$p$ groups we refer to \cite[\S~9.1]{rz:prof}.

\begin{defn}\label{def:semidirect prod}
Let $(G,\chi)$ be a cyclotomic pair and let $m\geq 1$ be an integer. The orientation $\chi$ induces an action of $G$ onto $\Z_p^m$, by 
$$g\bullet(x_1,\dots,x_m)=(\chi(g)x_1,\dots,\chi(g)x_m)\qquad\text{for }g\in G\text{ and }(x_1,\dots,x_m)\in\Z_p^m.$$
This in turn defines the \emph{cyclotomic semidirect product} $\Z_p^m\rtimes G$, by the rule 
$$g(x_1,\dots,x_m)g^{-1}=g\bullet(x_1,\dots,x_m).$$ The associated cyclotomic pair $(\Z_p^m\rtimes G, \chi\circ\pi)$, where $\pi:\Z_p^m\rtimes G\to G$ is the canonical projection,
is also called the \emph{cyclotomic semidirect product} (cf. \cite[\S~3]{efrat:small}).
\end{defn}

\subsection{Elementary type groups}\label{ssec:ET groups}
Any free pro-$p$ group is realizable (this is a consequence of the proposition in \cite[\S~4.8]{lubvan}). Some are actually also maximal pro-$p$ quotients of absolute Galois groups of local fields not containing primitive $p$\textsuperscript{th} roots of unity (cf.~\cite[Ch.~II, \S~5, Theorem 3]{serre}). On the other hand, maximal pro-$p$ quotients of absolute Galois groups of local fields that contain a primitive $p$\textsuperscript{th} root of unity are Demushkin or trivial (cf.~\cite[Ch.~II, \S~5, Theorem 4]{serre}).
It is not yet known whether all infinite Demushkin groups are realizable, nor whether there are infinite Demushkin groups which are realizable but not arising from $p$-adic fields.
Moreover, if $G_1$ and $G_2$ are realizable, so are also the free product $G_1\ast_p G_2$ and the semidirect products $\Z_p^m\rtimes G_1$ (cf.~\cite{efrhar}).
I.~Efrat conjectured that there are actually no other ways to obtain realizable finitely generated pro-$p$ groups (cf.~\cite{ido:ETC}, \cite{jacwar}). The original formulation is in terms of cyclotomic pairs:

\begin{conj}[Elementary Type Conjecture (ETC)]\label{conj:ETC}
The class $\mathcal{RFG}_p$ of realizable finitely generated pro-$p$ cyclotomic pairs is the smallest class of cyclotomic pairs such that
\begin{enumerate}
 \item[(a)] the pair $(1,1)$ consisting of the trivial group and the trivial orientation is in $\mathcal{RFG}_p$, as well as the pairs $(\Z_p,\chi)$ for any continuous homomorphism $\chi\colon \Z_p\to U_p^{(1)}$;
 \item[(b)] the pair $(\Z/2\Z,\chi)$ with $\mathrm{Im}(\chi)=\{\pm1\}$ is in $\mathcal{RFG}_p$, as well as any pair $(G_F(p),\chi_p)$
with $F$ a $p$-adic field satisfying Hypothesis~\ref{ass} and $\chi_p$ the canonical cyclotomic character of $G_F(p)$;
\item[(c)] if $(G_1,\chi_1),(G_2,\chi_2)\in\mathcal{RFG}_p$, then also the free product
$(G_1\ast_pG_2,\chi_1\ast_p\chi_2)$ is in $\mathcal{RFG}_p$;
\item[(d)] if $(G,\chi)\in\mathcal{RFG}_p$, then for any positive integer $m$ also the cyclotomic semi-direct product
$(\Z_p^m\rtimes G,\chi\circ\pi)$ is in $\mathcal{RFG}_p$.
\end{enumerate}
\end{conj}

The conjecture parallels an analogue conjecture for Witt rings. For a general discussion we refer to \cite{marshall:ETC}.

As long as we focus on notions that can be formulated independently of number theory,
we can work with a possibly bigger class:
\begin{defn}\label{defi:ETC}
The class $\mathcal{E}_p$ of $p$-\emph{elementary type} cyclotomic pairs is the smallest class of cyclotomic pairs such that
\begin{enumerate}
 \item[(a)] the pair $(1,1)$ consisting of the trivial group and the trivial orientation is in $\mathcal{RFG}_p$, as well as the pairs $(\Z_p,\chi)$ for any orientation $\chi\colon \Z_p\to U_p^{(1)}$;
 \item[(b)] any pair $(G,\chi)$, with $G$ a Demushkin group and $\chi$ its canonical orientation;
\item[(c)] if $(G_1,\chi_1),(G_2,\chi_2)\in\mathcal{E}_p$, then also the free product
$(G_1\ast_pG_2,\chi_1\ast_p\chi_2)$ is in $\mathcal{E}_p$;
\item[(d)] if $(G,\chi)\in\mathcal{E}_p$, then for any positive integer $m$ also the cyclotomic semi-direct product
$(\Z_p^m\rtimes G,\chi\circ\pi)$ is in $\mathcal{E}_p$.
\end{enumerate}

An \emph{elementary type} pro-$p$ group (\emph{ET group} for short) is a group $G$ appearing in a pair in $\mathcal{E}_p$.
\end{defn}


\section{Proof of Theorems A, B and C}\label{sec:A,B,C}
The proofs of Theorems~A and B proceed in accordance with the inductive definition of the class of ET groups (Definition \ref{defi:ETC}). In other words, we prove first that the theorems hold for free and Demushkin groups, and then that their validity is preserved under free and cyclotomic semidirect products. Theorem C will be proved as an intermediate step in \S \ref{ssec:demushkin}.

\subsection{Free pro-$p$ groups}
\label{ssec:free}

Let $S$ be a finitely generated free pro-$p$ group, with minimal generating set $\{x_1,\ldots,x_d\}$.
Let $\FppX$ be the algebra of formal power series in the non-commuting indeterminates $X=\{X_1,\ldots,X_d\}$ over $\F_p$.
It can be described as the completion of the free algebra $\F_p\gen X$ with respect to
the maximal two-sided ideal $I(X)=(X_1,\dots X_d)$.
As such, $\FppX$ is a topological, compact $\F_p$-algebra, and the closure of $I(X)$ in $\FppX$ is open.
Let $\FppX^\times$ denote the group of units of $\FppX$.
The \emph{Magnus morphism}, given by $x_i\mapsto1+X_i$, induces a monomorphism of profinite groups
\begin{equation}\label{eq:magnus}
 \mu:S\to1+I(X)\subseteq\FppX^\times,
\end{equation}
and an isomorphism of compact $\F_p$-algebras $\F_p\dbl S\dbr\cong\FppX$.
Moreover, one has $S_{(n)}=\{g\in S\mid \mu(g)-1\in I(X)^n\}$ for every $n\geq1$ (cf.~\cite[Sections 4.2, 7.6]{koch}). 
The graded object
\[ \gr\FppX=\bigoplus_{n\geq0}I(X)^n/I(X)^{n+1},\qquad\text{with }I(X)^0=\FppX, \]
is isomorphic to the non-commutative polynomial algebra $\FpX=T_\bullet(\Span(X))$.
Then the embedding $S\hookrightarrow\F_p\dbl S\dbr$, given by $g\mapsto g-1$,
induces a monomorphism of restricted Lie $\F_p$-algebras $L(S)\to\gr\FppX_L$, mapping the initial form $\overline{x_i}$ of $x_i$ to $X_i$. Since the image of $L(S)$ is the free restricted Lie algebra generated by $X$, and $\FpX$ is its universal restricted enveloping algebra,
we get isomorphisms of graded $\F_p$-algebras 
\[
\gr\F_p[\![S]\!]\cong\sU(L(S))\cong\FpX.
\]

On the other hand, the $\F_p$-cohomology of a free pro-$p$ group is the trivial algebra on a basis dual to $\{\overline{x_1},\ldots,\overline{x_d}\}$.
Therefore, Examples \ref{ex:quadratic} and \ref{ex:PBW} yield
\begin{thm}\label{thm:free koszul}
Let $S$ be a finitely generated free pro-$p$ group. Then the algebras $H^\bullet(S,\F_p)$ and $\gr\F_p\dbl S\dbr\cong \sU(L(S))$ are quadratic dual to each other and PBW. 
This establishes Theorems~A and B in the case of finitely generated free pro-$p$ groups.
\end{thm}

\subsection{Demushkin groups}\label{ssec:demushkin}
We stress that, in view of Theorem C, in the treatment of Demushkin groups in this subsection we do not require Hypothesis \ref{ass2}.
That is, we do not assume that the canonical orientation $\chi$ has $\mathrm{Im}(\chi)\subseteq 1+4\Z_2$ if $p=2$. 

If $G$ is a cyclic group of order $2$, then its $\F_2$-cohomology is isomorphic to the polynomial $\F_2$-algebra
in one variable, i.e., $H^1(G,\F_2)$ has order $2$ and
\begin{equation}\label{eq:cohomology C2}
 H^\bullet(G,\F_2)=T_\bullet(H^1(G,\F_2))\cong \F_2[X_1].
\end{equation}
On the other hand, for $G$ a cyclic group of order $2$, $G_{(n)}$ is trivial for $n\geq2$, so $\gr\F_2\dbl G\dbr$ is the trivial $\F_2$-algebra on a $1$-dimensional vector space:
\begin{equation}\label{eq:Z2}
\gr\F_2\dbl G\dbr\cong \F_2\oplus G,
\end{equation}
Hence by Examples~\ref{ex:quadratic} and \ref{ex:PBW} the two $\F_2$-algebras are quadratic dual to each other and PBW.

Assume now that $G$ is infinite. By definition, $H^\bullet(G,\F_p)$ is concentrated in degrees $0$, $1$ and $2$. More specifically, $H^1(G,\F_p)$ has an $\F_p$-basis $\{\chi_1,\dots,\chi_d\}$, with $d$ equal to the minimal number of generators of $G$, and $H^2(G,\F_p)$ has an $\F_p$-basis with just one element $\xi$. The generator $\xi$ and the cup product on $H^\bullet(G,\F_p)$ depend on the shape of a presentation of $G$ in accordance with Theorem \ref{thm:presentation demushkin}. Namely (cf.~\cite[Proposition 3.9.13]{nsw:cohm}),
\[
\xi=\chi_1\cup\chi_2=-\chi_2\cup\chi_1=\chi_3\cup\chi_4=-\chi_4\cup\chi_3=\dots=\chi_{d-1}\cup\chi_d=-\chi_d\cup\chi_{d-1}
\]
and $\chi_i\cup\chi_j=0$ in any other case for presentation \eqref{eq:demu1};
\[
\xi=\chi_1\cup\chi_1=\chi_2\cup\chi_3=\chi_3\cup\chi_2=\chi_4\cup\chi_5=\chi_5\cup\chi_4=\dots=\chi_{d-1}\cup\chi_d=\chi_d\cup\chi_{d-1}
\]
and $\chi_i\cup\chi_j=0$ in any other case for presentation \eqref{eq:demu2};
\[
\xi=\chi_1\cup\chi_1=\chi_1\cup\chi_2=\chi_2\cup\chi_1=\chi_3\cup\chi_4=\chi_4\cup\chi_3=\dots=\chi_{d-1}\cup\chi_d=\chi_d\cup\chi_{d-1}
\]
and $\chi_i\cup\chi_j=0$ in any other case for presentations \eqref{eq:demu3} and \eqref{eq:demu4}.

On the other hand, by \cite[Theorem~6.3]{jochen:massey}, $G$ is mild with respect to the $p$-Zassenhaus filtration. Then by \cite[Theorem~2.12]{jochen:massey}
\[\gr\FppG\cong\frac{\F_p\langle X_1,\dots,X_d\rangle}{\mathcal{R}},\]
where $\mathcal{R}$ is the two-sided ideal generated by the image of the initial form of the relator $r$ in $\F_p\langle X_1,\dots,X_d\rangle\cong\gr\F_p\dbml X_1,\dots,X_d\dbmr$. This image is
\begin{equation}\label{eq:demushkin initial form relator}
 \begin{array}{ll}
  [X_1,X_2]+[X_3, X_4]+\dots+[X_{d-1},X_d]&\mbox{for presentation \eqref{eq:demu1}}\\
X_1^2+[X_2,X_3]+\dots+[X_{d-1},X_d]&\mbox{for presentation \eqref{eq:demu2}}\\
X_1^2+[X_1,X_2]+\dots+[X_{d-1},X_d]&\mbox{for presentations \eqref{eq:demu3} and \eqref{eq:demu4}}.
 \end{array}
\end{equation}

In each of the 3 cases, by construction, $\{\chi_1,\dots,\chi_d\}$ and $\{X_1,\dots,X_d\}$ are dual bases.
Then the quadratic dual of $\gr\FppG$ is seen to be $H^\bullet (G,\F_p)$ by explicit computation.
Now $\gr\FppG$ is PBW by Lemma \ref{lem:one-relator PBW}, and consequently so is $H^\bullet (G)$ by Remark~\ref{rem:PBW}.
This completes the proof of the following.

\begin{thm}\label{thm:demushkin koszul}
Let $G$ be a Demushkin group.
Then the algebras $H^\bullet(G,\F_p)$ and $\gr\FppG\cong \sU(L(G))$ are quadratic dual to each other and PBW. 
This establishes Theorem~C, as well as Theorems~A and B in the case of Demushkin groups.
\end{thm}

\subsection{Free products}
\label{ssec:freeprod}

The $\F_p$-cohomology of the free pro-$p$ product of two pro-$p$ groups is described by the following
(cf.~\cite[Theorem~4.1.4]{nsw:cohm}).

\begin{prop}\label{propo:freeprod cohomology}
 Let $G_1$ and $G_2$ be finitely generated pro-$p$ groups, and set $G=G_1\ast_{p}G_2$.
 Then the inclusions $G_i\hookrightarrow G$ induce an isomorphism of graded $\F_p$-algebras
 \begin{equation}\label{eq:freeprod cohomology}
  \mathrm{res}_{G,G_1}^\bullet\oplus\mathrm{res}_{G,G_2}^\bullet\colon H^\bullet (G,\F_p) 
  \overset{\cong}{\longrightarrow} H^\bullet(G_1,\F_p)\sqcap H^\bullet(G_2,\F_p).
 \end{equation}
In particular, if both $H^\bullet(G_1,\F_p)$ and $H^\bullet(G_2,\F_p)$ are quadratic, so is $H^\bullet (G,\F_p)$.
\end{prop}

On the other hand, the graded $\F_p$-group algebra of the free pro-$p$ product
of two pro-$p$ groups is described by the following (see also \cite[\S~3]{lichtman:lie}).

\begin{prop}\label{lem:freeprodL}
 Let $G_1$, $G_2$ and $G$ be as above. Then there is an isomorphism of graded $\F_p$-algebras
\begin{equation}\label{eq:lemfreeprodL}
      \sU(L(G_1))\sqcup\sU(L(G_2))
      \overset{\cong}{\longrightarrow}\sU(L(G)).
\end{equation}
\end{prop}
\begin{proof}
Let $G^{\mathrm{abs}}$ denote the abstract free product of $G_1$ and $G_2$ as abstract, discrete groups. 
Then $G$ is the completion of $G^{\mathrm{abs}}$ with respect to the pro-$p$ topology
and one has the chain of inclusions $G_i\hookrightarrow G^{\mathrm{abs}}\hookrightarrow G$ for both $i=1,2$.
By \cite[Theorem~2]{lichtman:lie} the restricted Lie algebra of $G^{abs}$ is the free product of restricted
Lie algebras
\begin{equation}\label{eq:freeprod Labs}
 L(G^{\mathrm{abs}})=L(G_1)\ast L(G_2).
\end{equation}

Recall that, since $G$ is a finitely generated pro-$p$ group, its Zassenhaus filtration $(G_{(n)})_{n\geq1}$ is a topological neighborhood basis of 1 consisting of open normal subgroups of $G$ (cf., e.g., \cite[Theorem~7.11]{koch}).
Therefore, the inclusion $G^{\mathrm{abs}}\hookrightarrow G$ induces isomorphisms
of finite $p$-groups $G^{\mathrm{abs}}/G^{\mathrm{abs}}_{(n)}\cong G/G_{(n)}$ for every $n\geq1$, since $G$
is the pro-$p$ completion of $G^{abs}$.
Thus \eqref{eq:freeprod Labs} implies
\[\label{eq:freeprod L}
L(G_1)\ast L(G_2)=L(G^{\mathrm{abs}})\overset{\cong}{\longrightarrow} L(G).
\]
Now \eqref{eq:lemfreeprodL} follows from Proposition \ref{prop:U prod sum}.
\end{proof}

\begin{prop}\label{prop:freeprod all}
 Let $G_1$ and $G_2$ be two finitely generated pro-$p$ groups, and set $G=G_1\ast_p G_2$.
Assume that $H^\bullet(G_1,\F_p)$ and $H^\bullet(G_2,\F_p)$ are quadratic $\F_p$-algebras.
\begin{itemize}
 \item[(i)] If for $i=1,2$ $H^\bullet(G_i,\F_p)^!\cong\gr\F_p\dbl G_i\dbr$,
then $H^\bullet(G,\F_p)^!\cong \gr\FppG$.
 \item[(ii)] If for $i=1,2$ $H^\bullet(G_i,\F_p)$, resp.~$\gr\F_p\dbl G_i\dbr$, are PBW algebras,
then also $H^\bullet(G,\F_p)$, resp.~$\gr\FppG$, is PBW.
\end{itemize}
\end{prop}
\begin{proof}
\begin{itemize}
 \item[(i)] follows from Proposition \ref{propo:freeprod cohomology}, Proposition \ref{lem:freeprodL}, Theorem \ref{thm:LG grFpG} and Remark \ref{rem:de Morgan}.
 \item[(ii)] follows from Proposition \ref{propo:freeprod cohomology}, Proposition \ref{lem:freeprodL}, Theorem \ref{thm:LG grFpG} and Example \ref{ex:PBW}.
 \end{itemize}
\end{proof}

\subsection{Semidirect products}\label{ssec:semidirect}
Let $(G_0,\chi_0)$ be a finitely generated cyclotomic pair.
The $\F_p$-cohomology of a semidirect product $(\Z_p^m\rtimes G_0, \chi_0\circ\pi)$ as in Definition \ref{def:semidirect prod}
is described in \cite[Theorem~3.1, Corollary~3.4, Theorem 3.6]{wadsworth:cohomology}.
Actually, in that paper the author uses valuations on fields,
but a careful checking reveals that one can extract from his arguments a proof of the following proposition
that relies only on the group-theroretic data we have and the Hochschild-Serre spectral sequence (cf., e.g., \cite[Theorem~3.13]{qw:cyc}).


\begin{defn}\label{def:wadsworth extension}
Let $A_\bullet$ be a graded-commutative quadratic $\F_p$-algebra, with space of generators $V$ and space of relators $\Omega$, and let $\{t,a_1,\dots,a_d\}$ be a set which linearly spans $V$, with $t$ a distinguished element such that $t+t=0$.
Let $\{x_j\mid j\in J\}$ be a set of distinct symbols not in $A$. The \emph{twisted extension} of $A_\bullet$ by $J$ is the quadratic $\F_p$-algebra $A[J;t]_\bullet$ with space of generators $\Span_{\F_p}\{t,a_1,\dots,a_d,x_j\mid j\in J\}$ and space of relators 
$$\Span_{\F_p}(\Omega\cup \{x_ix_j+x_jx_i,x_jt+tx_j,x_ja_k+a_kx_j,x_j^2-tx_j\mid i,j\in J, k=1,\dots,d\}).$$
\end{defn}
(The notion of twisted extension was introduced in \cite[Definition~1.8]{wadsworth:cohomology}.)

\begin{rem}\label{rem:AJt}
 Let $A_\bullet$ be as in Definition~\ref{def:wadsworth extension}.
 If $t=0$ then $A[J;t]_\bullet\cong A_\bullet\otimes^{-1}\Lambda_\bullet(W)$, where $W=\Span_{\F_p}\{x_j\mid j\in J\}$.
 Note that if $p\neq2$ then necessarily $t=0$.
\end{rem}

\begin{prop}\label{prop:fibreprod cohomology}
 Let $(G_0,\chi_0)$ be a finitely generated cyclotomic pair and $m$ be a positive integer, and set
 $(G,\chi)=(\Z_p^m\rtimes G_0, \chi_0\circ\pi)$.
 Then the inflation map $H^\bullet(G_0,\F_p)\to H^\bullet(G,\F_p)$ and the restriction map
 $H^\bullet(G,\F_p)\to H^\bullet(\Z_p^m,\F_p)$
 induce an isomorphism of $\F_p$-algebras
 \begin{equation}\label{eq:fibreprodH}
 H^\bullet(G,\F_p)\overset{\cong}{\longrightarrow}H^\bullet(G_0,\F_p)[J,t],
\end{equation}
with $J=\{1,\dots,m\}$.
In particular, if $H^\bullet(G_0,\F_p)$ is quadratic, then so is $H^\bullet(G,\F_p)$.
\end{prop}
If $G_0$ is the maximal pro-$p$ quotient of the absolute Galois group of a field $F$ satisfying Hypothesis \ref{ass}, the element $t$ corresponds to the square class of
$-1\in F$ in Bloch-Kato isomorphism, so $t=0$ if $p$ is odd, or if $p=2$ and $F$ satisfies also Hypothesis \ref{ass2} (cf.~\cite[Examples 1.12]{wadsworth:cohomology}). More generally, $t=0$ for all cyclotomic pairs $(G,\chi)$ in the class $\mathcal{E}_p$ (see Definition \ref{defi:ETC}), satisfying $\mathrm{Im}(\chi)\subseteq1+4\Z_2$ in case $p=2$, and so
 \begin{equation}\label{eq:fibreprodHt0}
 H^\bullet(G,\F_p)\overset{\cong}{\longrightarrow}H^\bullet(G_0,\F_p)\otimes^{-1}\Lambda_\bullet(W),
\end{equation}
where $W$ is as in Remark~\ref{rem:AJt}.

Let $k\in\N\cup\{\infty\}$ be such that $\mathrm{Im}(\chi_0)=1+p^k\Z_p\subseteq \Z_p^\times$, with the convention that $p^\infty=0$.
We set $Z=\Z_p^m$, and we shall use the multiplicative notation for it.
\begin{lem}\label{lem:gamma fprod}
Let $(G_0,\chi_0)$, $(G,\chi)$ and $Z$ be as above.
\begin{itemize}
 \item[(i)] For every $s\geq1$ one has $(ZG_0)^{p^s}=Z^{p^s}\cdot G_0^{p^s}$.
 \item[(ii)] For every $n\geq2$ one has $\gamma_n(G)=Z^{p^{(n-1)k}}\cdot \gamma_n(G_0)$. 
 \item[(iii)] For every $n\geq1$, $G_{(n)}=\left(Z^{p^{\lceil \log_p(n)\rceil}}\right)\rtimes (G_0)_{(n)}$. 
\end{itemize}
\end{lem}

\begin{proof}
Every element $g$ of $G$ can be written as $g=z\cdot w$, with $z\in Z$ and $w\in G_0$.
Clearly, for any $s\geq1$ one has the inclusions $Z^{p^s}\subseteq(ZG_0)^{p^s}$ and $G_0^{p^s}\subseteq(ZG_0)^{p^s}$,
so that $(ZG_0)^{p^s}\supseteq Z^{p^s}G_0^{p^s}$.
On the other hand, for any $z\in Z$ and $w\in G_0$ one has
\[
(zw)^{p^s}=z^{1+\chi_0(w)+\ldots+\chi_0(w)^{p^s-1}}\cdot w^{p^s}=
\begin{cases}
z^{p^s}\cdot w^{p^s} & \text{if }\chi_0(w)\!=\!1, \\ 
z^0\cdot w^{p^s} & \text{if }p\!=\! 2\text{ and }\chi_0(w)\!=\!-1, \\ 
 z^{\frac{\chi_0(w)^{p^s}-1}{\chi_0(w)-1}}\cdot w^{p^s} & \text{otherwise.}
\end{cases}\]
Clearly, the exponent of $z$ is divisible by $p^s$ in the first two cases. 
Suppose that we are in the last case. 
 Let $v_p$ be the $p$-adic valuation. One can check that, for any $\alpha\in p\Z_p$ (with $\alpha\in 4\Z_2$ if $p=2$) and $n\in \N$,
\[
v_p((1+\alpha)^n-1)=v_p(\alpha)+v_p(n)
\]
(this is sometimes called the Lifting The Exponent Lemma).
Therefore, if $\chi_0(w)\in 1+p^k\Z_p$ (with $k\geq 2$ if $p=2$), then by writing $\chi_0(w)=1+\alpha$, one has
\[
v_p\left(\frac{\chi_0(w)^{p^s}-1}{\chi_0(w)-1}\right)=v_p((1+\alpha)^{p^s}-1)-v_p(\alpha)=v_p(p^s)=s.
\]
Now suppose $p=2$ and $\chi_0(w)\in 1+2\Z_2$ but $\chi_0(w)\not\in 1+4\Z_2$. Then one may write  $\chi_0(w)=1+2\alpha$ for some $\alpha\in \Z_2^\times$. One has
\[
\begin{aligned}
v_2\left(\frac{\chi_0(w)^{p^s}-1}{\chi_0(w)-1}\right)&= v_2((1+2\alpha)^{2^s}-1) -v_2(2\alpha)\\
&=v_2((1+4\alpha+4\alpha^2)^{2^{s-1}}-1)-1\\
&=v_2(4\alpha+4\alpha^2)+v_2(2^{s-1}) -1\\
&\geq 2+(s-1)-1=s,
\end{aligned}
\]
and hence $\dfrac{\chi_0(w)^{2^s}-1}{\chi_0(w)-1}$ is also divisible by $2^s$.

In all three cases the exponent of $z$ is divisible by $p^s$, thus $(zw)^{p^s}\in Z^{p^s}G_0^{p^s}$.
This yields statement (i).

For arbitrary $z_1,z_2\in Z$ and $w_1,w_2\in G_0$ one has 
\begin{equation}\label{eq:commcalc}
\begin{array}{lll}
[z_1w_1,z_2w_2] &\!\!=\!\!& w_1^{-1}z_1^{-1}w_2^{-1}z_2^{-1}z_1w_1z_2w_2\\
&\!\!=\!\!& [w_1,z_1]z_1^{-1}w_1^{-1}w_2^{-1}z_1z_2^{-1}w_1w_2z_2[z_2,w_2]\quad\mbox{($Z$ is abelian)}\\
&\!\!=\!\!& [w_1,z_1][z_1,w_2w_1][w_1,w_2][w_1w_2,z_2][z_2,w_2].\\
\end{array}
\end{equation}

Since the action of $G_0$ on $Z$ via $\chi_0$ yields $[z,w],[w,z]\in Z^{p^k}$ for all $z\in Z^{p^k}$ and $w\in G_0$, we have $\gamma_2(G)\subseteq Z^{p^k}\cdot\gamma_2(G_0)$.
On the other hand, clearly $\gamma_2(G_0)\subseteq\gamma_2(G)$, and for any $z\in Z$ and $w\in G_0$ such that $\chi_0(w)=1+p^k$
one has $z^{p^k}=[w^{-1},z^{-1}]\in\gamma_2(G)$.
Therefore $\gamma_2(G)= Z^{p^k}\cdot\gamma_2(G_0)$.

Note that $\gamma_n(G_0)\subseteq \mathrm{Ker}(\chi)$ for every $n\geq2$,
so that the action of $\gamma_n(G_0)$ on $Z$ is trivial.

Consequently, for any $n$ and $k\in\N$, $Z^{p^{(n-1)k}}\cdot\gamma_n(G_0) = Z^{p^{(n-1)k}}\times\gamma_n(G_0)$. It follows that for all $z\in Z$, all $w\in G_0$, all $s\in Z^{p^{(n-1)k}}$ and all $g\in \gamma_n(G_0)$
\[
[zw,sg]=[zw,g]g^{-1}[zw,s]g=[zw,g][zwg,s]\in [G,\gamma_n(G_0)]\cdot[G,Z^{p^{(n-1)k}}].
\] 
Statement (ii) follows by induction on $n$, since, assuming $\gamma_n(G)=Z^{p^{(n-1)k}}\cdot \gamma_n(G_0)$, 
\[
\begin{split} \gamma_{n+1}(G) &= \left[G,Z^{p^{(n-1)k}}\cdot\gamma_n(G_0)\right] \\
&= [G,\gamma_n(G_0)]\cdot\left[G,Z^{p^{(n-1)k}}\right]\\
&= [G_0,\gamma_n(G_0)]\cdot(Z^{p^{(n-1)k}})^{p^k}=Z^{p^{nk}}\cdot\gamma_{n+1}(G_0).
\end{split}
\]

Finally, by Equality \eqref{eq:Lazard} one has
\begin{equation}\label{eq:zass for semidirect}
\begin{split}
G_{(n)} &= G^{p^{\lceil\log_p(n)\rceil}}\cdot\prod_{\substack{i\geq2\\ip^h\geq n}}\gamma_i(G)^{p^h} \\
 &= \left(Z^{p^{\lceil\log_p(n)\rceil}}\cdot G_0^{p^{\lceil\log_p(n)\rceil}}\right)\cdot\prod_{\substack{i\geq2\\ip^h\geq n}}\left(Z^{p^{(i-1)k}}\cdot \gamma_i(G_0)\right)^{p^h}.
 \end{split}\end{equation}
 Since $Z^{p^{(i-1)k}}\cdot \gamma_i(G_0)=Z^{p^{(i-1)k}}\times \gamma_i(G_0)$ for every $i\geq2$, one has 
\[\begin{split}
 \prod_{\substack{i\geq2\\ip^h\geq n}}\left(Z^{p^{(i-1)k}}\cdot \gamma_i(G_0)\right)^{p^h} &=
 \prod_{\substack{i\geq2\\ip^h\geq n}}Z^{p^{(i-1)k+h}}\times\prod_{\substack{i\geq2\\ip^h\geq n}}\gamma_i(G_0)^{p^h}\\
 &=Z^{p^\eta}\times \prod_{\substack{i\geq2\\ip^h\geq n}}\gamma_i(G_0)^{p^h},
 \end{split}\]
where $\eta=\min\{(i-1)k+h\mid i\geq2,ip^h\geq n\} $, as $Z^{p^{(i-1)k}}\subseteq Z^{p^\eta}$.
  Moreover, $p^\eta\geq n$ for every $n\geq2$, thus $\eta\geq \lceil\log_p(n)\rceil$, and $Z^{p^\eta}\subseteq Z^{p^{\lceil\log_p(n)\rceil}}$, so that from \eqref{eq:zass for semidirect} one obtains 
 \[\begin{split}
 G_{(n)} &= 
 Z^{p^{\lceil\log_p(n)\rceil}}\cdot G_0^{p^{\lceil\log_p(n)\rceil}}\cdot Z^{p^\eta}\cdot\prod_{\substack{i\geq2\\ip^h\geq n}} \gamma_i(G_0)^{p^h}\\
 &= Z^{p^{\lceil\log_p(n)\rceil}}\cdot Z^{p^\eta}\cdot\left[Z^{p^\eta},G_0^{p^{\lceil\log_p(n)\rceil}}\right]\cdot G_0^{p^{\lceil\log_p(n)\rceil}}\prod_{\substack{i\geq2\\ip^h\geq n}} \gamma_i(G_0)^{p^h}\\
 &=Z^{p^{\lceil\log_p(n)\rceil}}\cdot G_0^{p^{\lceil\log_p(n)\rceil}}\prod_{\substack{i\geq2\\ip^h\geq n}} \gamma_i(G_0)^{p^h},
\end{split}\]
which gives statement (iii).
%
%
%
\end{proof}

\begin{prop}\label{prop:fprod LandU}
Let $(G_0,\chi_0)$, $(G,\chi)$ and $Z$ be as above, and assume $\mathrm{Im}(\chi_0)\subseteq1+4\Z_2$ if $p=2$.
One has an isomorphism of graded $\F_p$-algebras
\begin{equation}\label{eq:fprodU}
\F_p[X]\otimes^1\gr\F_p\dbl G_0\dbr\overset{\cong}{\longrightarrow}\gr\FppG,
\end{equation}
with $X=\{X_1,\ldots,X_m\}$.
\end{prop}

\begin{proof}
Since $Z$ is a free abelian pro-$p$ group of rank $m$,
the restricted Lie algebra $L(Z)$ is a free abelian restricted Lie algebra on $m$ generators.
In particular, one has $L_n(Z)$ trivial if $n$ is not a power of $p$, whereas $L_{p^h}(Z)=Z^{p^h}/Z^{p^{h+1}}$ for $h\geq0$.
Thus, $\sU(L(Z))$ is isomorphic to the commutative polynomial algebra $\F_p[X]$, with $X=\{X_1,\ldots,X_m\}$.

For $n\geq1$, Lemma~\ref{lem:gamma fprod} implies that
\[     \frac{G_{(n)}}{G_{(n+1)}} =
\frac{Z^{p^{\lceil\log_p(n)\rceil}}\rtimes (G_0)_{(n)}}{Z^{p^{\lceil\log_p(n+1)\rceil}}\rtimes (G_0)_{(n+1)}}
=\begin{cases}
  \dfrac{Z^{p^h}}{Z^{p^{h+1}}}\times \dfrac{(G_0)_{(n)}}{(G_0)_{(n+1)}}  & \text{for }n=p^h, \\
   \dfrac{(G_0)_{(n)}}{(G_0)_{(n+1)}} & \text{otherwise}.
                            \end{cases}   \]
Moreover for every $z\in Z_{(i)}$ and $w\in (G_0)_{(j)}$ one has $[z,w]\in Z_{(n)}$ with $n>i+j$ because of the action
induced by $\chi_0$, so that in $L(G)$ one has $[\bar z,\bar w]=0$ (Lie bracket).
Therefore, the restricted Lie algebra $L(G)$ is the direct sum $L(Z)\oplus L(G_0)$ of restricted Lie algebras.
Hence, 
\[\begin{split}
   \gr\FppG &=\sU(L(Z)\oplus L(G_0))\\ &\cong \sU(L(Z))\otimes^1\sU(L(G_0))=\F_p[X]\otimes^1\gr\F_p\dbl G_0\dbr
  \end{split}\]
by Proposition~\ref{prop:U prod sum}.
\end{proof}

Summing up, Propositions~\ref{prop:fibreprod cohomology} and \ref{prop:fprod LandU}, together with Remark~\ref{rem:AJt}, the isomorphism~\eqref{eq:fibreprodHt0}, and Example~\ref{ex:PBW}, imply the following.

\begin{thm}\label{thm:fibreprod all}
 Let $(G_0,\chi_0)$ be a finitely generated cyclotomic pair (satisfying $\mathrm{Im}(\chi_0)\subseteq1+4\Z_2$ if $p=2$) and
 $(G,\chi)=(\Z_p^m\rtimes G_0, \chi_0\circ\pi)$ for some positive integer $m$.
 Assume that $H^\bullet(G_0,\F_p)$ is a quadratic $\F_p$-algebra.
\begin{itemize}
 \item[(i)] If $H^\bullet(G_0,\F_p)^!\cong\gr\F_p\dbl G_0\dbr$,
 then $H^\bullet(G,\F_p)^!\cong \gr\FppG$.
 \item[(ii)] If $H^\bullet(G_0,\F_p)$, resp.~$\gr\F_p\dbl G_0\dbr$, is a PBW algebra,
then also $H^\bullet(G,\F_p)$, resp.~$\gr\FppG$, is PBW.
\end{itemize}
\end{thm}

In fact, it is also true that a general twisted extension --- that is, with $t$ not necessarily $0$ (and thus with possibly $\{\pm1\}\subseteq\mathrm{Im}(\chi)$, in the setting of cyclotomic pro-$p$ pairs of elementary tipe) --- of a PBW $\F_2$-algebra is PBW, as shown by the following result of independent interest.

\begin{prop}\label{prop:twisted extension PBW}
Let $A_\bullet$ be a graded-commutative quadratic $\F_2$-algebra endowed with space of generators $V=\Span_{\F_2}\{t,a_1,\dots,a_d\}$ and space of relators $\Omega$, let $J=\{1,\dots,m\}\subset\N$ and let $\{x_j\mid j\in J\}$ be a set of distinct symbols not in $A$. Suppose that $\{t,a_1,\dots,a_d\}$ is a set of PBW generators of $A_\bullet$ with respect to the degree-lexicographic order associated to
$
t\preceq_1a_1\preceq_1\dots\preceq_1a_d
$.
Then $\{t,a_1,\dots,a_d, x_1,\dots,x_m\}\!$ is a set of PBW generators of the twisted extension $A[J;t]_{\!\bullet}$ with respect to the degree-lexicographic order associated to
\[
t\preceq_1a_1\preceq_1\dots\preceq_1a_d\preceq_1x_1\preceq_1\dots\preceq_1x_m.
\]
\end{prop}
\begin{proof}
We can always choose a normalized basis of $\Omega$ containing all the commutativity relators $a_it+ta_i$ and $a_ja_i+a_ia_j$ ($1\leq i<j\leq d$), plus possibly other relators. The rewriting rules in $A[J;t]_\bullet$ are those in $A_\bullet$ and (since $-1=1$)
\[
\begin{array}{ll}
x_ja_k\rightsquigarrow a_kx_j & (j\in J, k\in\{1,\dots,d\});\\
x_jt\rightsquigarrow tx_j & (j\in J);\\
x_jx_i\rightsquigarrow x_ix_j & (i,j\in J, i<j);\\
x_j^2\rightsquigarrow tx_j & (j\in J).
\end{array}
\]
The critical monomials in $A[J;t]_\bullet$ are those in $A_\bullet$ and
\begin{enumerate}
\item $x_j^3$;
\item $x_j^2x_i\; (i<j)$;
\item $x_jx_i^2\; (i<j)$;
\item $x_jx_ix_h\;(h<i<j)$;
\item $x_j^2t$;
\item $x_j^2a_k$;
\item $x_jx_it\; (i<j)$;
\item $x_jx_ia_k\; (i<j)$;
\item $x_jb_1b_2$ ($b_1b_2$ a leading monomial in $\Omega$).
\end{enumerate}

The critical monomials in $A_\bullet$ are confluent by hypothesis. The rewriting graphs of the new ones follow:

\
\hspace*{-\parindent}
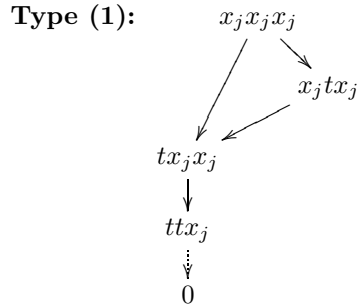
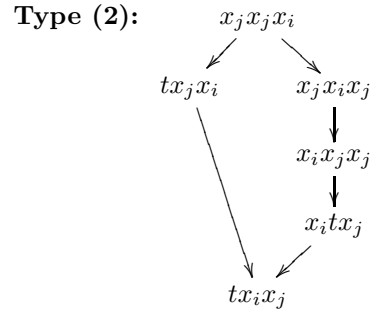
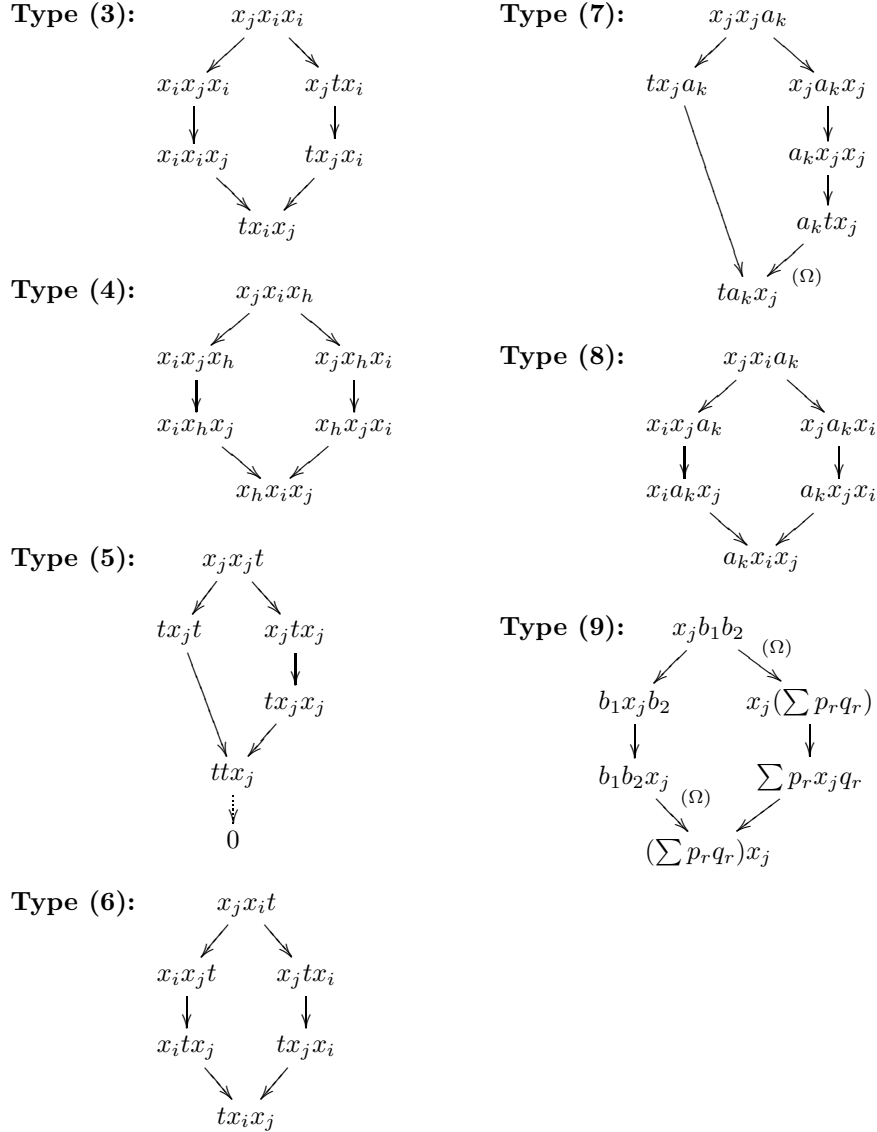
\begin{minipage}{\textwidth}
\begin{multicols}{2}
\begin{description}
\item[Type (1)] \hskip-.5em$\xymatrix@R-14pt@C-30pt{&x_jx_jx_j\ar[ddl]\ar[dr]&\\
&&x_jtx_j\ar[dll]\\
tx_jx_j\ar[d]&&\\
ttx_j\ar@{.>}[d]&&\\
0}$
\item[Type (2)] \hskip-.5em$\xymatrix@R-14pt@C-30pt{&x_jx_jx_i\ar[dl]\ar[dr]&\\
tx_jx_i\ar[dddr] && x_jx_ix_j\ar[d]\\
&&x_ix_jx_j\ar[d]\\
&&x_itx_j\ar[dl]\\
& tx_ix_j &}$
\end{description}
\end{multicols}
\vspace*{.6\baselineskip}
\begin{multicols}{2}
\begin{description}
\item[Type (3)] \hskip-.5em$\xymatrix@R-14pt@C-30pt{&x_jx_ix_i\ar[dl]\ar[dr]&\\
x_ix_jx_i\ar[d] && x_jtx_i\ar[d]\\
x_ix_ix_j\ar[dr] && tx_jx_i\ar[dl]\\
&tx_ix_j&}$
\item[Type (4)] \hskip-.5em$\xymatrix@R-14pt@C-30pt{&x_jx_ix_h\ar[dl]\ar[dr]&\\
x_ix_jx_h\ar[d] && x_jx_hx_i\ar[d]\\
x_ix_hx_j\ar[dr] && x_hx_jx_i\ar[dl]\\
&x_hx_ix_j&}$
\end{description}
\end{multicols}
\vspace*{.6\baselineskip}
\begin{multicols}{2}
\begin{description}
\item[Type (5)] \hskip-.5em$\xymatrix@R-14pt@C-30pt{&x_jx_jt\ar[dl]\ar[dr]&\\
tx_jt\ar[ddr] && x_jtx_j\ar[d]\\
&& tx_jx_j\ar[dl]\\
&ttx_j\ar@{.>}[d]&\\
&0&}$
\item[Type (6)] \hskip-.5em$\xymatrix@R-14pt@C-30pt{&x_jx_ja_k\ar[dl]\ar[dr]&\\
tx_ja_k\ar[dddr] && x_ja_kx_j\ar[d]\\
&&a_kx_jx_j\ar[d]\\
&&a_ktx_j\ar^-{(\Omega)}[dl]\\
& ta_kx_j &}$
\end{description}
\end{multicols}
\vspace*{.6\baselineskip}
\begin{multicols}{2}
\begin{description}
\item[Type (7)] \hskip-.5em$\xymatrix@R-14pt@C-30pt{&x_jx_it\ar[dl]\ar[dr]&\\
x_ix_jt\ar[d] && x_jtx_i\ar[d]\\
x_itx_j\ar[dr] && tx_jx_i\ar[dl]\\
&tx_ix_j&}$
\item[Type (8)] \hskip-.5em$\xymatrix@R-14pt@C-30pt{&x_jx_ia_k\ar[dl]\ar[dr]&\\
x_ix_ja_k\ar[d] && x_ja_kx_i\ar[d]\\
x_ia_kx_j\ar[dr] && a_kx_jx_i\ar[dl]\\
&a_kx_ix_j&}$
\end{description}
\end{multicols}
\vspace*{.6\baselineskip}
\begin{multicols}{2}
\begin{description}
\item[Type (9)] \hskip-18pt$\xymatrix@R-14pt@C-40pt{&x_jb_1b_2\ar[dl]\ar^-{(\Omega)}[dr]&\\
b_1x_jb_2\ar[d] && x_j(\sum p_rq_r)\ar[d]\\
b_1b_2x_j\ar^-{(\Omega)}[dr] && \sum p_rx_jq_r\ar[dl]\\
&(\sum p_rq_r)x_j&}$
\end{description}
\end{multicols}
\end{minipage}

\pagebreak
The dotted arrows in Types (1) and (5) denote a further reduction in case $tt\in \Omega$. The label $(\Omega)$ in Types (7) and (9) denotes the application of a rewriting rule in $A_\bullet$. In particular, in (7) $a_kt\rightsquigarrow ta_k$ is always a rewriting rule according to our choice of a normalized basis of $\Omega$. In (9), the rewriting rule associated to the leading monomial $b_1b_2$ is supposed to be $b_1b_2\rightsquigarrow\sum p_rq_r$; here the possibility of a monomial rewriting rule $b_1b_2\rightsquigarrow 0$ is not excluded: in this case the graph can be simplified, but confluency still holds.
\end{proof}

 Theorem~\ref{thm:fibreprod all}~(ii) and Proposition~\ref{prop:twisted extension PBW}, together with Proposition~\ref{prop:fibreprod cohomology}, yield Theorem~A for cyclotomic semidirect products of cyclotomic pro-$p$ pairs of elementary type.


%

\section{Proof of Theorem D}\label{sec:D}

In view of Theorems~A and B, in order to prove Theorem~D it suffices to show that
the maximal pro-$p$ Galois group $G_F(p)$, together with the cyclotomic character $\chi_p$, is a cyclotomic pro-$p$
pair of elementary type, for any field $F$ as in the statement of Theorem~D.

First, with deal with the cases (a)--(d) and (f) of the theorem,
then we give a more detailed treatment for $p$-rigid fields and Pythagorean formally real fields.

\begin{rem}
 By Kummer theory one has an isomorphism of discrete $\F_p$-vector spaces $F^\times/(F^\times)^p\cong H^1(G_F(p),\F_p)$
  (see, e.g., \cite[Theorem~6.2 and proof]{koch}, see also \cite[Chapter~6, Theorem~3]{artintate} for a description of infinite Kummer theory),
and thus $G_F(p)$ is finitely generated if and only if the quotient $F^\times/(F^\times)^p$ is finite.
\end{rem}

\begin{prop}
Let $F$ be a field as in items (a) or (b) of Theorem~D.
Then $(G_F(p),\chi_p)$ is a cyclotomic pro-$p$ pair of elementary type. 
\end{prop}

\begin{proof}
In the cases when $F$ is a finite field, or $F$ is PAC field one knows that for every finite Galois extension $K$ of $F$ the norm map $K^*\to F^*$ is surjective. Since this condition implies that the Brauer group of each finite Galois extension of $F$ is trivial, $H^2(G_F(p),\F_p) =0$, i.e., $G_F(p)$ is a free pro-$p$ group, and hence $(G_F(p),\chi_p)$ is of elementary type.

For finite fields this is the well-known Chevalley Theorem (\cite[Theorem~25]{shatz}). 
If $F$ is PAC then this follows from \cite[Corollary~11.6.8]{friedjarden}. Note that an algebraic extension of a PAC field is again PAC (cf.~\cite[Corollary~11.2.5]{friedjarden}).

 Finally, if $F$ is an extension of a PAC field of relative trascendence degree 1, then the cyclotomic pro-$p$ pair $(G_F(p),\chi_p)$ is of elementary type by \cite[Corollary~5.7]{efrat:PAC}.
\end{proof}

\begin{prop}
Let $F$ be a field as in cases (c), (d) or (f) of Theorem~D.
Then $(G_F(p),\chi_p)$ is a cyclotomic pro-$p$ pair of elementary type. 
\end{prop}

\begin{proof}
For $F$ a field of the type (c), the claim follows from \cite[Main Theorem]{efrat:Q}.

If $F$ is an extension of trascendence degree 1 of a non-Archimedean local field, the claim follows from \cite[Theorem~6.6 and Theorem~6.10]{efrat:local}. 
If $F=\C(t)$ then $G_F(p)$ is a free pro-$p$ group by Tsen's Theorem (cf.~\cite[Corollary on p.~109]{shatz}). If $F=\R(t)$, then  $G_F(p)$ is real-free (cf.~\cite{haranjarden,haranjarden2}), so that $G_F(p)$ is free for $p$ odd, and $(G_F(2),\chi_2)$ is of elementary type.

Finally, for $F$ a field of the type (f) and $p=2$, the claim follows from \cite{KK:globalfields}.
See also \cite[\S~10]{marshall:ETC}.
\end{proof}

\subsection{Rigid fields} \label{ssec:prigid}

For a field $F$ and an odd prime $p$, let $\sqrt[p^\infty]{\K}$ denote the extension of $\K$ obtained by adjoining the elements $\sqrt[p^k]{a}$ for all $k\geq1$ and $a\in\K$. A field $F$ satisfying the Standing Hypothesis~\ref{ass} and such that $\K^\times/(\K^\times)^p$ is finite is said to be \emph{$p$-rigid} if $\K(p)=\sqrt[p^\infty]{\K}$ (cf.~\cite[\S~4]{cmq:fast}).
 In particular, the following three statements are equivalent (cf.~\cite[\S~3]{cmq:fast}):
 \begin{itemize}
  \item[(i)] $\K$ is $p$-rigid;
  \item[(ii)] $(G_{\K}(p),\chi_{\K,p})\cong\Z_p^m\rtimes(G_0,\chi_0)$ for some $m\geq0$
and some cyclotomic pair $(G_0,\chi_0)$ such that either $G_0\cong \Z_p$ and $\chi_0$ is injective,
or $(G_0,\chi_0)$ is trivial;
  \item[(iii)] $H^\bullet(G_{\K}(p),\F_p)\cong\Lambda_\bullet(H^1(G_{\K}(p),\F_p))$.
 \end{itemize}
 Then, by Proposition~\ref{prop:fprod LandU}, one deduces from the characterization (ii) above that, for $\K$ $p$-rigid,
\[ \gr\F_p\dbl G_{\K}(p)\dbr\cong\F_p[X]\otimes^1\gr\F_p\dbl G_0\dbr\cong\begin{cases}
                                                        \F_p[X]\otimes^1\gr\F_p[X_0]=\F_p[Y], &\text{or} \\ \F_p[X]                
                                                                       \end{cases}\]
with $X=\{X_1,\ldots,X_d\}$ for some $d\geq1$, $X_0$ a single variable, and $Y=X\cup\{X_0\}$ That is, $\gr\F_p\dbl G_{\K}(p)\dbr$ is a free abelian graded algebra.

Therefore, both algebras $H^\bullet(G_{\K}(p),\F_p)$ and $\gr\F_p\dbl G_{\K}(p)\dbr$ are PBW by Example~\ref{ex:PBW}~(b), and they are quadratically dual to each other by Example~\ref{ex:quadratic2}~(b).
This proves Theorem~D~(e).

Moreover, one has the following.

\begin{coro}\label{coro:rigid gr}
 Let $p$ an odd prime number and let $\K$ be a field containing a primitive $p$\textsuperscript{th} root of unity and such that $\K^\times/(\K^\times)^p$ is finite.
 Then $\K$ is $p$-rigid if, and only if, $\gr\F_p\dbl G_{\K}(p)\dbr$ is an abelian graded $\F_p$-algebra.
\end{coro}

\begin{proof}
We have already observed that if $\K$ is $p$-rigid, then $\gr\F_p\dbl G_{\K}(p)\dbr$ is isomorphic to a symmetric algebra. For the converse, if $\gr\F_p\dbl G_{\K}(p)\dbr$ is abelian, then $L(G_\K(p))$ is an abelian restricted Lie algebra.
In particular, the commutator map
\[
\xymatrix{[\slot,\slot]\colon\dfrac{{G_\K(p)}}{{G_\K(p)}_{(2)}}\times\dfrac{{G_\K(p)}}{{G_\K(p)}_{(2)}}\ar@{->>}[r]
&\dfrac{{G_\K(p)}_{(2)}}{{G_\K(p)}_{(3)}}},\]
which is surjective since $p\neq2$ and thus $G_F(p)^p\subseteq G_{F}(p)_{(3)}$, is the zero map. Hence ${{G_\K(p)}_{(2)}}/{{G_\K(p)}_{(3)}}=0$.
Since the $\F_p$-dual of ${{G_\K(p)}_{(2)}}/{{G_\K(p)}_{(3)}}$ is the kernel
of the map 
\begin{equation}\label{eq:cuprigid2}
 H^1({G_\K(p)},\F_p)\wedge H^1({G_\K(p)},\F_p)\longrightarrow H^2({G_\K(p)},\F_p)
\end{equation}
induced by the cup product (cf.~\cite[Example~4.2~(1)]{EM2}),
we have that such kernel is trivial, namely, \eqref{eq:cuprigid2} is an isomorphism.
Since $H^\bullet(G_F(p),\F_p)$ is quadratic, this isomorphism in degree two implies that $H^\bullet(G_F(p),\F_p)\cong \Lambda_\bullet(H^1(G_{\K}(p),\F_p))$,
so the assertion follows by (iii) in the above characterization of $p$-rigid fields.
\end{proof}


\subsection{Pythagorean formally real fields}\label{ssec:pita}
In this subsection we assume that fields contain a primitive $2$\textsuperscript{nd} root of $1$, which is equivalent to having characteristic different from $2$.

Recall that a field $F$ is said to be \emph{Pythagorean} if $F^2+F^2=F^2$.
The family $\mathcal{PFR}$ of Pythagorean formally real fields with finitely many square classes consists of those Pythagorean fields $F$ such that $-1$ is not a square in $F$ and $|F^\times/(F^\times)^2|<\infty$.
Particular elements of $\mathcal{PFR}$ are \emph{Euclidean} fields, that is fields $F$ such that 
$$(F^\times)^2+(F^\times)^2=(F^\times)^2\not\ni-1\qquad \text{and} \qquad F^\times=(F^\times)^2\cup-(F^\times)^2$$ (cf.~\cite{becker}): for any Euclidean field $F$ (for instance, $F=\R$), one has $G_F(2)=\Z/2\Z$, with $\mathrm{Im}(\chi_2)=\{\pm1\}$.

Pythagorean fields $F$ which are not formally real are uninteresting from our point of view, as for such fields one has $F^\times=(F^\times)^2$ (cf.~\cite[p.~255]{lam}), and so $G_F(2)=\{1\}$.
This establishes trivially Theorem~D~(g) for Pythagorean fields which are not formally real.

 By the results in \cite{craven,jacob1}, one may construct the $\F_2$-cohomology algebra $H^\bullet(G_F(2),\F_2)$ of the maximal pro-$2$ Galois group $G_F(2)$ of a Pythagorean formally real field $F$, starting from the $\F_2$-cohomology algebras of $\Z_2$ and of the cyclic group of order $2$, employing direct sums and twisted extensions of quadratic algebras (cf.~Example~\ref{ex:quadratic2}~(a) and Definition~\ref{def:wadsworth extension} respectively).
Moreover, let $\mathcal{EE}_2$ denote the smallest class of cyclotomic pro-$2$ pairs of elementary type containing the pair $(\Z/2\Z,\chi)$, 
with $\mathrm{Im}(\chi)=\{\pm1\}$, and closed with respect to the free products of cyclotomic pairs and the cyclotomic semidirect products with $\Z_2$ --- 
note that for $G\in\mathcal{EE}_2$ and for a finite product $Z=\prod_{i=1}^m\Z_2$ (written multipicatively), the action of $G$ on $Z$, which defines the semidirect product $Z\rtimes G$, is given by
\begin{equation}\label{eq:PythCyclotomicAction}
\sigma^{-1}z\sigma=z^{-1}\qquad \mbox{for any } \sigma\in G\smallsetminus\{1\} \mbox{ such that } \sigma^2=1 \mbox{ and any } z\in Z.
\end{equation}
Observe that the action is analogous to the cyclotomic action mentioned at the beginning of \S~\ref{sec:ET}.
In fact, a property of Pythagorean fields $F$ is that their extension $F(\mu_{2^\infty})$
with all roots of unity of order a power of $2$ coincides with $F(\sqrt{-1})$
(cf. \cite{BEK:orderings,Becker:pythagorean}).
Each $\sigma$ as in \eqref{eq:PythCyclotomicAction} is sent to $-1$ by the $2$-adic cyclotomic character
\[
\chi_2\colon G_{F_0}(2)\to\{-1,+1\}\cong\Gal(F(\sqrt{-1})/F)=\Gal(F(\mu_{2^\infty})/F).
\]
The following characterization of maximal pro-$2$ Galois groups of Pythagorean formally real fields is proved in \cite{minac}.

\begin{thm}\label{thm:PFR}
A cyclotomic pro-2 pair $(G,\chi)$ is in $\mathcal{EE}_2$ if and only if $G$ is realizable as the maximal pro-$2$ Galois group of a field in $\mathcal{PFR}$.
\end{thm}

Using the description of the family $\mathcal{PFR}$ given by Theorem \ref{thm:PFR}, we can take advantage of previous results. In particular, $H^\bullet(\Z/2\Z,\F_2)$ has been studied at the beginning of \S \ref{ssec:demushkin} and the cohomology of free products has been addressed in \S \ref{ssec:freeprod}.
As regards semidirect products, for any field $F_0\in\mathcal{PFR}$ and any finite product $Z=\prod_{i=1}^m\Z_2$, the $\F_2$-cohomology of $Z\rtimes G_{F_0}(2)$ is a twisted extension $H^\bullet(G_{F_0}(2),\F_2)[J,t]$, with $J=\{1,\dots,m\}$ and $t$ corresponding to the square class of $-1$ in Bloch-Kato isomorphism.
Then the proof of Theorem~D~(g) is completed by applying Propositions \ref{prop:fprod LandU} and \ref{prop:twisted extension PBW}.

\section{Proof of Theorems E and F}\label{sec:pairings and duals}
Most results of this section are valid in general, but for simplicity and coherence with the rest of the paper we assume groups to be finitely generated.

Let $G$ be a finitely generated pro-$p$ group.
A \emph{minimal presentation} of $G$ is a short exact sequence of pro-$p$ groups
\begin{equation}\label{eq:presentation}
\xymatrix{ 1\ar[r] & R\ar[r] & S\ar[r] & G\ar[r] & 1,} 
\end{equation}
where $S$ is a free pro-$p$ group and $R\subseteq S_{(2)}$, that is, the epimorphism
$S\to G$ induces an isomorphism $S/S_{(2)}\cong G/G_{(2)}$.
For such a presentation, a \emph{set of defining relations} is a minimal subset of $R$ which generates $R$
as a normal subgroup of $S$.
If $S$ is free on the set $\{x_1,\dots,x_d\}$ and $\{r_1,\dots,r_m\}$ is a set of defining relations, we also write
\begin{equation}\label{eq:gen rel pres}
G=\gen{x_1,\dots,x_d\mid r_1,\dots,r_m}.
\end{equation}

The sequence \eqref{eq:presentation} induces the five-term exact sequence in cohomology
\[\xymatrix@C-6pt{ 0\ar[r] & H^1(G,\F_p)\ar[r]^-{\mathrm{inf}_{G,S}^1} & H^1(S,\F_p)\ar[r]^-{\mathrm{res}_{S,R}^1} 
& H^1(R,\F_p)^G\ar[r]\ar[r]^-{\trg} & H^2(G,\F_p) \ar[r]^-{\mathrm{inf}_{G,S}^2} & H^2(S,\F_p).} \]
As $S$ is free, $H^2(S,\F_p)=0$ (cf.~\cite[Ch.~I, \S~3.4]{serre}). As $S$ and $G$ have the same minimal number of generators, the map $\mathrm{inf}_{G,S}^1\colon H^1(G,\F_p)\to H^1(S,\F_p)$ is an isomorphism. Then $\trg\colon H^1(R,\F_p)^G\to H^2(G,\F_p)$ is an isomorphism as well.
We set $V=H^1(G,\F_p)$.

\subsection{Pairings}\label{ssec:pairings}
Recall that a pairing $A\times B\to\F_p$ of vector spaces over $\F_p$ is said to be \emph{perfect}
if it induces isomorphisms $A\cong B^*$ and $B\cong A^*$.
The transgression map induces the natural perfect pairing
\begin{equation}\label{eq:pairing R}
 \langle\slot,\slot\rangle_R \colon \frac{R}{R^p[R,S]}\times H^2(G)\longrightarrow \F_p,
\end{equation}
which is given by $\langle \bar r,\alpha\rangle=({\trg}^{-1}(\alpha))(r)$ (cf.~\cite[\S~III.9, page 233]{nsw:cohm}).

On the other hand, one has an isomorphism of finite $\F_p$-vector spaces $G/G_{(2)}\cong S/S_{(2)}\cong V^*$.
Let $\{x_1,\ldots,x_d\}$ be a minimal generating system for $S$.
Recall from \S \ref{ssec:free} that we may identify the tensor algebra ${T}_\bullet(V^*)$
with the graded free algebra $\FpX$, with $X=\{X_1,\ldots,X_d\}$, and the quotient $S_{(2)}/S_{(3)}$
embeds in $T^2(V^*)$, via the monomorphism of spaces given by $[x_i,x_j]\mapsto X_iX_j-X_jX_i$,
and also $x_i^2\mapsto X_i^2$ if $p=2$.

We may identify $V^*\otimes V^*$ with $(V\otimes V)^*$ via the pairing 
\[\begin{split}
&(V^*\otimes V^*)\times (V\otimes V)\longrightarrow \F_p,\\ &(f\otimes g, u\otimes v)\longmapsto(-1)f(u)g(v)
  \end{split}\] 
in accordance with Koszul sign convention (cf.~\cite[\S~1.5.3]{lodval}).
Thus, we may consider $S_{(2)}/S_{(3)}$ as a subspace of $(V\otimes V)^*$ via the composition $S_{(2)}/S_{(3)}\hookrightarrow T^2(V^*)\cong(V^{\otimes2})^*$, and one has
the pairing
\begin{equation}\label{eq:pairing SxVxV}
 \langle\slot,\slot\rangle_S \colon S_{(2)}/S_{(3)} \times (V\otimes V) \longrightarrow \F_p
\end{equation}
(note that this pairing is not perfect).

Since $R\subseteq S_{(2)}$, one has that $R^p[R,S]\subseteq S_{(3)}$, and one may define the morphism
$f\colon R/R^p[R,S]\to S_{(2)}/S_{(3)}$ induced by the inclusion.
Then the following holds, where we denote by $\cup$ the map $V\otimes V\to H^2(G,\F_p)$ induced by the cup product via 
bilinearity.

\begin{prop}\label{prop:commutative diagram} 
Let $G$ be a finitely generated pro-$p$ group with minimal presentation \eqref{eq:presentation}.
The diagram of pairings
\begin{equation}
\label{eq:commutative diagram pairings}
 \xymatrix@R=1.1truecm{ S_{(2)}/S_{(3)}\times (V\otimes V)\ar@<6ex>[d]_-{\cup}\ar[rr] && \F_p\ar@{=}[d] \\
R/R^p[R,S]\times H^2(G,\F_p)\ar[rr]\ar@<6ex>[u]^-{f} && \F_p }
\end{equation}
is commutative, i.e., $\langle f(\bar{r}),\alpha\rangle_S = \langle \bar{r}, \cup(\alpha)\rangle_R$,
for every $r\in R$, $\alpha\in V\otimes V$.
\end{prop}

\begin{proof}
Let $\{\chi_1,\ldots,\chi_d\}\subseteq H^1(S,\F_p)$ be a basis dual to $\{x_1,\ldots, x_d\}$, so $\chi_i(x_j)=\delta_{ij}$. 
Then $\{\chi_i\otimes \chi_j, 1\leq i, j\leq d\}$ is a basis for $H^1(S,\F_p)\otimes H^1(S,\F_p)$.
Therefore, it is enough to show that
$\langle f(\bar{r}), \chi_k\otimes \chi_l \rangle = \langle \bar{r}, \chi_k\cup \chi_l\rangle$
for every $r\in R$ and $1\leq k,l\leq d$.

Note that $r$  may be uniquely written as
\[
\label{decomposition modulo S3}
r=
\begin{cases}
\displaystyle \prod_{i=1}^d x_i^{2a_i} \prod_{ i<j}[x_i,x_j]^{b_{ij}}\cdot r', &\text{if } p=2,\\
\displaystyle  
\displaystyle  \prod_{i<j}[x_i,x_j]^{b_{ij}} \cdot r', &\text {if } p\not=2,
\end{cases}
 \]
where $a_i,b_{ij}\in \{0,1,\ldots,p-1\}$ and $r'\in S_{(3)}$ (cf.~\cite[Proposition~1.3.2]{vogel:thesis}). 
Then one can see that 
\[  f(\bar{r})= 
 \begin{cases}
 \sum_{i}a_i \bar{x}_i^2 +  \sum_{i<j}^d b_{ij}(\bar{x}_i\bar{x}_j - \bar{x}_j\bar{x}_i) &\text {if } p=2,\\
 \sum_{i<j}^d b_{ij}(\bar{x}_i\bar{x}_j - \bar{x}_j\bar{x}_i) &\text {if } p\not=2.
 \end{cases} \]
Hence
\[ \langle f(\bar{r}), \chi_k\otimes \chi_l \rangle =
 \begin{cases} 
-b_{kl} &\text{if } k<l,\\
b_{kl} &\text{if } k>l,\\
 0 &\text{if } k=l \text{ and } p\neq2\\
-a_k &\text{if } k=l \text{ and } p=2.
 \end{cases}\]
On the other hand
\[  \langle \bar{r}, \cup(\chi_k\otimes \chi_l)\rangle= {\rm trg}^{-1}(\chi_k\cup \chi_l)(r). \]
The result then follows from \cite[Proposition~1.3.2]{vogel:thesis}.
\end{proof}

\begin{prop}\label{prop:perfectpairing wedge}
 The pairing \eqref{eq:pairing SxVxV}  and the commutative diagram \eqref{eq:commutative diagram pairings}
induce the commutative diagrams of perfect pairings
\begin{equation}\label{eq:pairings S wedgeV}
 \xymatrix{ S_{(2)}/S_{(3)}\times \Lambda_2(V)\ar@<6ex>[d]_-{\cup}\ar[r] & \F_p\ar@{=}[d] &
S_{(2)}/S_{(3)}\times \mathcal{S}_2(V)\ar@<6ex>[d]_-{\cup}\ar[r] & \F_2\ar@{=}[d] \\
R/R^p[R,S]\times H^2(G,\F_p)\ar[r]\ar@<6ex>[u]^-{f} & \F_p &
R/R^2[R,S]\times H^2(G,\F_2)\ar[r]\ar@<6ex>[u]^-{f} & \F_2 }
\end{equation}
(the left-hand side one if $p$ is odd, the right-hand side one if $p=2$).
\end{prop}

\begin{proof}
The identification ${T}_\bullet(V^*)=\F_p\langle X\rangle$ induces an embedding of the quotient $S_{(2)}/S_{(3)}$
in the space of homogeneous polynomials of degree two.

Assume first that $p$ is odd. It is well-known (cf.~\cite[Lemma 5.3]{jochen:massey}) that
the set of commutators $\{[X_i,X_j],1\leq i<j\leq d\}$ is a basis for $S_{(2)}/S_{(3)}$.
On the other hand, the set
\[
\left\{ \chi_i\otimes\chi_j-\chi_j\otimes\chi_i\mid1\leq i< j\leq d\}\cup\{\chi_i\otimes\chi_j+\chi_j\otimes\chi_i ,1\leq i\leq j\leq d\right\}
\]
is a basis for $V\otimes V$.
Since 
\begin{eqnarray*}
&& \langle [X_i,X_j],\chi_h\otimes\chi_k-\chi_k\otimes\chi_h\rangle_S=2\delta_{ih}\delta_{jk},\\
&& \langle [X_i,X_j],\chi_h\otimes\chi_k+\chi_k\otimes\chi_h\rangle_S=0
\end{eqnarray*}
for every $1\leq i<j\leq d$ and $1\leq h\leq k\leq d$, the orthogonal space $(S_{(2)}/S_{(3)})^\perp$
in $V\otimes V$ is the subspace generated by $\chi_i\otimes\chi_j+\chi_j\otimes\chi_i$, $1\leq i\leq j\leq d$,
and the pairing \eqref{eq:pairing SxVxV} induces the perfect pairing
\[ S_{(2)}/S_{(3)}\times \frac{V\otimes V}{(S_{(2)}/S_{(3)})^\perp}=
S_{(2)}/S_{(3)}\times \Lambda_2(V)\longrightarrow\F_p.\]
Note that $\dim(S_{(2)}/S_{(3)})=\dim(\Lambda_2(V))=d(d-1)/2$.

In case $p=2$, then $\{[X_i,X_j],X_i^2,1\leq i<j\leq d\}$ is a basis for $S_{(2)}/S_{(3)}$ (cf.~\cite[Lemma~5.3]{jochen:massey}).
After replacing $\Lambda_2(V)$ with $\mathcal{S}_2(V)$, an analogous argument 
shows that the pairing 
\[ S_{(2)}/S_{(3)}\times \frac{V\otimes V}{(S_{(2)}/S_{(3)})^\perp}=
S_{(2)}/S_{(3)}\times \mathcal{S}_2(V)\longrightarrow\F_2\]
is perfect.

Since the cup product is graded-commutative, the commutativity of diagram \eqref{eq:pairings S wedgeV}
follows by the commutativity of diagram \eqref{eq:commutative diagram pairings}.
\end{proof}

\begin{thm}\label{thm:perfectpairing}
 Let $G$ be a pro-$p$ group with minimal presentation \eqref{eq:presentation}.
The following are equivalent.
\begin{enumerate}
\item[(i)] The pairing \eqref{eq:pairing R} induces a perfect pairing $RS_{(3)}/S_{(3)} \times H^2(G,\F_p)\to \F_p$.
\item[(ii)] The cup product $\cup\colon H^1(G,\F_p)\otimes H^1(G,\F_p)\to H^2(G,\F_p)$ is surjective.
\item[(iii)] One has the equality $R\cap S_{(3)}=R^p[R,S]$.
\end{enumerate}
\end{thm}

\begin{proof}
 Let $W$ be the subspace of $H^2(G,\F_p)$ generated by the cup products of elements of $V$.
From \eqref{eq:pairings S wedgeV} one obtains the diagram of pairings
\begin{equation}
\label{eq:commutative diagram 3floors}
 \xymatrix@R=0.7truecm{ S_{(2)}/S_{(3)}\times \Lambda_2(V)\ar@{->>}@<6ex>[d]\ar[rr] && \F_p\ar@{=}[d] \\
RS_{(3)}/S_{(3)}\times W\ar@{^{(}->}@<6ex>[d]\ar@{^{(}->}@<5ex>[u]\ar[rr] && \F_p\ar@{=}[d] \\
R/R^p[R,S]\times H^2(G,\F_p)\ar[rr]\ar@{->>}@<5ex>[u] && \F_p }
\end{equation}
(with $\mathcal{S}_2(V)$ instead of $\Lambda_2(V)$ if $p=2$) which commutes by Proposition~\ref{prop:perfectpairing wedge}.
Also, the top and bottom rows are perfect pairings.
Therefore, by linear algebra considerations, also the middle row is a perfect pairing,
and the claim follows by dimension counting.
\end{proof}

From Theorem~\ref{thm:perfectpairing} one may recover the following, which is also a particular consequence of 
\cite[Examples~8.3~(1)]{EM2} (there the authors deal with a more general situation).

\begin{coro}\label{coro:RcapF3} 
Let $G$ be a pro-$p$ group with minimal presentation \eqref{eq:presentation}
and with quadratic $\F_p$-cohomolgy.
Then $R\cap S_{(3)}=R^p[R,S]$, and all initial forms of the defining relations are of degree two.
\end{coro}

\begin{proof}
The equality $R\cap S_{(3)}=R^p[R,S]$ follows directly from Theorem~\ref{thm:perfectpairing}.
If $\{r_1,\ldots,r_m\}$ is a set of defining relations of $G$, then $\{\overline{r_1},\ldots,\overline{r_m}\}$ is a basis of the space $R/R^p[R,S]$ (here $\overline{r_i}$ denotes the coset $r_i\cdot R^p[R,S]$ for every $i=1,\ldots,m$).
Since $R/R^p[R,S]=R/R\cap S_{(3)}$, one has $r_i\notin S_{(3)}$ for every $i=1,\ldots,m$, that is, the initial form of $r_i$ has degree $2$.
\end{proof}

\begin{rem}\label{rem:CEM}
In the case $G$ is the maximal pro-$p$ Galois group $G_{\K}(p)$ of a field $\K$ containing a root of unity of order $p$
then the three equivalent statements of Theorem~\ref{thm:perfectpairing} hold.
Moreover, in this case Corollary~\ref{coro:RcapF3} is a well-known result 
(cf.~\cite[Theorem~8.3]{cem:quot} and \cite[Corollary 3.5]{Ng}, and \cite{MSp} for $p=2$).
\end{rem}





\subsection{Quadratic duals}

Let $G$ be a finitely generated pro-$p$ group, and assume that its $\F_p$-cohomology algebra is quadratic.
Then we may write $H^\bullet(G,\F_p)$ as the quotient 
\begin{equation}\label{eq:quad cohom}
H^\bullet(G,\F_p)=\frac{{T}_\bullet(V)}{(\Omega)}, \qquad V=H^1(G,\F_p),\quad \Omega\subseteq V\otimes V.
\end{equation}
In this case it is possible to describe the quadratic dual of $H^\bullet(G,\F_p)$ in terms
of the initial forms of the relations.

\begin{thm}(Theorem E) \label{thm:quaddual} 
Let \eqref{eq:presentation} be a minimal presentation of a finitely generated pro-$p$-group $G$,
and assume that $H^\bullet(G,\F_p)$ is quadratic as in \eqref{eq:quad cohom}.
Then one has the equality $\Omega^\perp=RS_{(3)}/S_{(3)}$.
In particular, Theorem E holds: if $d$ is the minimal number of generators of $G$ and $X=\{X_1,\ldots,X_d\}$, one has an isomorphism of quadratic $\F_p$-algebras
\begin{equation}\label{eq:quadratic duality thm}
 H^\bullet(G,\F_p)^!\cong \frac{\F_p\langle X\rangle}{\mathcal{R}},
\end{equation}
where $\mathcal{R}$ is the two-sided ideal generated by the initial forms of the defining relations of $G$.
\end{thm}
\begin{proof}
Recall that one has the chain of inclusions \[RS_{(3)}/S_{(3)}\subseteq S_{(2)}/S_{(3)}\subseteq (V\otimes V)^*.\]
By Proposition~\ref{prop:commutative diagram} and Corollary \ref{coro:RcapF3} the diagram of pairings
\[
\xymatrix@C-8pt{\gen{\slot,\slot}_1:&(V\otimes V)^*\times(V\otimes V) \ar[rr] \ar@{->>}@<6ex>[d]_{\cup} && \F_p \\
 \gen{\slot,\slot}_2:&RS_{(3)}/S_{(3)}\times H^2(G) \ar@{^{(}->}@<6ex>[u] \ar[rr] && \F_p\ar@{=}[u]}
 \]
commutes, and by Theorem~\ref{thm:perfectpairing} also the bottom row is a perfect pairing.
If $\alpha\in\Omega$ then one has
$\langle\bar r,\alpha\rangle_1=\langle\bar r,\cup(\alpha)\rangle_2=\langle\bar r,0\rangle_2=0$
for every $\bar r\in RS_{(3)}/S_{(3)}$.
Thus, $RS_{(3)}/S_{(3)}$ is contained in the orthogonal space $\Omega^\perp$.
On the other hand, one has 
\[\dim(RS_{(3)}/S_{(3)})=\dim(H^2(G))=\dim(V\otimes V)-\dim(\Omega),\]
so that $RS_{(3)}/S_{(3)}=\Omega^\perp$.
Then the isomorphism \eqref{eq:quadratic duality thm} follows by the identification ${T}_\bullet(V^*)=\F_p\langle X\rangle$.
\end{proof}

For the following definition we borrow the terminology from \cite[\S~4]{KLM}.
\begin{defn}\label{def:quadratic presentation}
Let $G$ be a finitely generated pro-$p$ group with minimal presentation \eqref{eq:gen rel pres}.
\begin{itemize}
 \item[(i)] The presentation is called \emph{quadratic} (with respect to the Zassenhaus filtration) if
the initial forms $\rho_1,\dots,\rho_m$ of the defining relations are of degree two.
 \item[(ii)] The presentation is called \emph{quadratically defined} if it is quadratic and
 $\gr\FppG$ is isomorphic to $\FpX/\mathcal{R}$, with $X=\{X_1,\ldots,X_d\}$ and $\mathcal{R}=(\rho_1,\ldots,\rho_m)$.
\end{itemize}
\end{defn}


The Zassenhaus filtration of the free pro-$p$ group $S$ induces the filtration $\{R\cap S_{(n)}\}$,
with $n\geq2$.
Let $I(R)_{\ind}$ be the graded object induced by such filtration --- 
i.e., $$I(R)_{\ind}=\bigoplus_{n\geq2}(R\cap S_{(n)})/(R\cap S_{(n+1)}).$$ 
Then $I(R)_{\ind}$ is a restricted ideal of the restricted Lie algebra $L(F)$, and one has the short exact sequence
of restrcted Lie algebras
\begin{equation}\label{eq:inducedalgebras presentation}
 \xymatrix{ 0\ar[r]& I(R)_{\ind}\ar[r] & L(S)\ar[r] & L(G)\ar[r] &0 }, 
\end{equation}
i.e., $G_{(n)}=S_{(n)}/R\cap S_{(n)}$ for each $n\geq1$ (cf. \cite[Proof of Theorem~2.12]{jochen:massey}).
In general for a finitely generated pro-$p$ group $G$ with a quadratic presentation
the induced ideal $I(R)_{\ind}$ and the restrcted ideal $\mathfrak{r}$ generated by the initial forms of the relations
may be different. 

\begin{lem}\label{prop:compare ideals}
 Let $G$ be a finitely generated pro-$p$ group with quadratic presentation \eqref{eq:presentation}.
Then $I(R)_{\ind}\supseteq\mathfrak{r}$.
\end{lem}

\begin{proof}
Since $G$ has a quadratic presentation, the initial forms $\overline{r_i}$ of the defining relations of $G$ are of degree $2$.
In particular, $\overline{r_i}\in R/R\cap S_{(3)}$ for every $i\in\{1,\ldots,m\}$.
Since $R/R\cap S_{(3)}=(I(R)_{\ind})_2$, and since $\mathfrak{r}$ is the restricted ideal generated by the $\overline{r_i}$'s, the claim follows.
\end{proof}

Thus, a finitely presented pro-$p$ group $G$ with a quadratic presentation is quadratically defined
if, and only if, $\mathfrak{r}=I(R)_{\ind}$.

Theorem~\ref{thm:quaddual} implies Theorem~F.


\begin{proof}[Proof of Theorem~F]
Let $\mathcal{R}'$ and $\mathcal{R}$ be the two-sided graded ideals of $\FpX$ generated by $I(R)_{\ind}$ and $\mathfrak{r}$ respectively. By Theorem \ref{thm:quaddual},
\[
H^\bullet(G,\F_p)^!\cong\frac\FpX{\mathcal{R}}.
\]
On the other hand, by Proposition \ref{prop:presentation U} applied to the sequence \eqref{eq:inducedalgebras presentation},
\[
\gr\FppG\cong\frac\FpX{\mathcal{R}'}.
\]
Since by Lemma \ref{prop:compare ideals} $\mathfrak{r}\subseteq I(R)_{\ind}$, it follows that $\mathcal{R}\subseteq\mathcal{R}'$, so there is a graded algebra epimorphism
\[
H^\bullet(G,\F_p)^!\cong\frac\FpX{\mathcal{R}}\twoheadrightarrow\frac\FpX{\mathcal{R}'}\cong\gr\FppG.
\]
Since the elements of degree $2$ in $\mathfrak{r}$ and $I(R)_{ind}$ coincide, being the elements of $RS_{(3)}/S_{(3)}$, the previous epimorphism is an isomorphism in degrees $0$, $1$ and $2$.

If $G$ has a quadratically defined presentation, then by definition $\mathcal{R}=\mathcal{R}'$ and $\gr\FppG\cong\FpX/{\mathcal{R}}$.
\end{proof}

%

%
%
\begin{rem}
Let $p$ be an odd prime number and $G$ a finitely generated pro-$p$ group.
We have $H^1(G,\F_p)=H^1(G/G_{(2)},\F_p)$ and we shall identify (see \S \ref{ssec:cohom of pro-p})
\[
H^1(G,\F_p)^*= G/G_{(2)}.
\]

Let ${\mathfrak {lie}}(H^1(G,\F_p)^*)$ be the free restricted Lie algebra on the $\F_p$-vector space $H^1(G,\F_p)^*$. Then we have a natural restricted Lie algebra epimorphism
\[
\varphi_G\colon {\mathfrak {lie}}(H^1(G,\F_p)^*) \twoheadrightarrow L(G),
\] 
which restricts to the identity in degree $1$, and to the commutator map
\[
[ \slot , \slot]\colon G/G_{(2)} \wedge G/G_{(2)} \to G_{(2)}/G_{(3)}
\]
in degree $2$.
Let 
\[
\delta_G \colon H^2(G,\F_p)^* \to H^1(G,\F_p)^* \wedge H^1(G,\F_p)^* 
\]
be the map dual to the cup product $\cup \colon H^1(G,\F_p)\wedge H^1(G,\F_p) \to H^2(G,\F_p)$. From \cite[page 268]{Hill}, we have the following exact sequence
\[
1 \longrightarrow (G_{(2)}/G_{(3)})^* \stackrel{d}{\longrightarrow}H^1(G,\F_p)\wedge H^1(G,\F_p) \stackrel{\cup}{\longrightarrow} H^2(G,\F_p),
\]
where $d$ is dual to the commutator map $[\slot ,\slot ]\colon G/G_{(2)} \wedge G/G_{(2)} \to G_{(2)}/G_{(3)}$.  Thus $\mathrm{Im} (\delta_G)= \ker ([\slot , \slot])$. Hence we have a natural epimorphism
\[
\Phi_G\colon {\mathfrak {lie}}(H^1(G,\F_p)^*)/(\mathrm{Im}(\delta_G)) \twoheadrightarrow L(G).
\]
Following \cite{SW}, the \emph{holonomy restricted Lie algebra} ${\mathfrak H}(G)$ of $G$ is the holonomy restricted Lie algebra of the cohomology ring $H^\bullet (G,\F_p)$, that is,
\[
{\mathfrak H}(G)={\mathfrak H}(H^\bullet(G,\F_p)): = {\mathfrak {lie}}(H^1(G,\F_p)^*)/(\mathrm{Im}(\delta_G)).
\]
We also say that $G$ is ($p$-)\emph{graded-formal} if the natural epimorphism 
\[\Phi_G\colon {\mathfrak H}(G)\twoheadrightarrow L(G)
\]
is an isomorphism of graded restricted Lie algebras. Note also that, by \cite[Proposition 3.4]{SW}, one has
\[
{\sU}({\mathfrak H}(G)) = \overline{H^\bullet(G,\F_p)}^!,
\] 
where $\overline{H^\bullet(G,\F_p)}$ is the \emph{quadratic closure} of $H^\bullet(G,\F_p)$, defined as 
\[
\overline{H^\bullet(G,\F_p)}=\Lambda^\bullet(H^1(G,\F_p))/(\ker\cup).
\]
By Jennings's Theorem \ref{thm:LG grFpG}, the natural epimorphism $\Phi_G$ induces a natural epimorphism
\[
\overline{H^\bullet(G,\F_p)}^! \twoheadrightarrow \gr\FppG.
\]
Our Theorem B says that, for $p$ odd, any elementary type pro-$p$ group $G$ is $p$-graded formal.
 \end{rem}

\bibliographystyle{amsalpha}
\bibliography{koszul1}


\end{document}